\documentclass[
    a4paper,
    english,
    oneside]{article}

\usepackage[utf8]{inputenc}
\usepackage[T1]{fontenc}
\usepackage{expl3}
\usepackage[english]{babel}
\usepackage[a4paper,top=2.5cm,bottom=2.5cm,left=3cm,right=3cm]{geometry}

\usepackage{%
    amssymb,
    amsfonts,
    amsthm,
    mathtools,
    amsmath}

\input{./macros}

\usepackage[
    backend=biber,
    style=alphabetic,
    doi=false,
    isbn=false,
    url=true,
    date=year,
    maxbibnames=99,
    giveninits=true,
      ]{biblatex}

\AtEveryBibitem{\clearfield{series}}
\AtEveryBibitem{\clearfield{labeladdress}}

\definecolor{LinkBlue}{cmyk}{1,0.50,0,0}
\usepackage[
    hidelinks,
    colorlinks=true,
    unicode,
    linkcolor=LinkBlue,
    citecolor=LinkBlue,
    urlcolor=LinkBlue]{hyperref}

\title{From Nilspace Theory to Simplicial Homotopy Theory}
\author{Noa Bihlmaier\thanks{
    \textsc{Mathematisches Institut, Universität Bonn, Endenicher Allee 60, 53115 Bonn, Germany} \\
    \indent\textit{Email address:} \href{mailto:nobi@math.uni-bonn.de}{\texttt{nobi@math.uni-bonn.de}}}
        \and
        Nick Ruoff\thanks{
    \textsc{Mathematisches Institut, Universität Bonn, Endenicher Allee 60, 53115 Bonn, Germany} \\
    \indent\textit{Email address:} \href{mailto:ruoff@math.uni-bonn.de}{\texttt{ruoff@math.uni-bonn.de}}}}

\begin{document}
\maketitle
\thispagestyle{empty}
\begin{abstract}
We construct a faithful functor from the category of cubespaces, arising in structured ergodic theory and nilspace theory, to the category of (condensed) simplicial sets.
This functor translates the weak structure theorem of a nilspace
to a Postnikov tower of the corresponding (condensed) simplicial set.
In particular, nilspaces are mapped to Kan complexes and the structure groups of the nilspace
correspond to the simplicial homotopy groups of the Kan complex.
This allows us to translate topological and nilspace-theoretical Host-Kra theory to (simplicial) homotopy theory.

The functor is constructed via pullback along a special cosimplicial cubical set,
whose combinatorial properties may be of independent interest.
\end{abstract}

\section{Introduction}

Even though it is a well-known fact among nilspace theorists that there are obvious similarities between the theory of nilspaces and simplicial homotopy theory, no precise connection has been established yet \cite{Manners2023}.
The purpose of the present text is to establish such a connection and make it a precise mathematical statement.
Our main result is the following.
\begin{theorem*}
    There exists a limit preserving functor $K^\ast\colon\cCond\to\CondAni\otimes\sSet$ from condensed cubical sets to condensed simplicial sets which is faithful on the full subcategory of concrete cubespaces,
such that if $X$ is a concrete fibrant cubespace, the following holds.
\begin{enumerate}[(i)]
	\item The condensed simplicial set $K^\ast X$ is a Kan complex.
    \item The condensed structure groups $\pi_n(X,x_0)$ agree with the condensed homotopy groups $\pi_n(K^\ast X,x_0)$.
\end{enumerate}
\end{theorem*}

\subsection{A short outline of the construction}
Using the usual $(\infty)$-categorical machinery applied to condensed sets,
the condensed/topological statement reduces easily to the following discrete version, see Section \ref{sec:enrich}.

\begin{repeatresult}{theorem}{thm:main}
There exists a limit preserving functor $K^\ast\colon\cSet\to\sSet$ which is faithful on the full subcategory of concrete cubesets, such that if $X$ is a concrete fibrant cubeset, the following holds.
\begin{enumerate}[(i)]
	\item The simplicial set $K^\ast X$ is a Kan complex.
    \item The structure groups $\pi_n(X,x_0)$ agree with the homotopy groups $\pi_n(K^\ast X,x_0)$.
\end{enumerate}
\end{repeatresult}

The functor $K^*$ is constructed via pullback with an explicit cosimplicial cubical set $K$.
To define $K$, we need to translate the combinatorial properties of the cube category $\bbox$ and of the simplex category $\bbDelta$.
For this, we introduce the concept of \emph{formal cubes} which make it possible to abstractly manipulate the cubes representing $n$-simplices.

\begin{repeatresult}{definition}{def:formal-cubes}
A \emph{formal $n$-dimensional cube} with labels in $[m]$ is a function $X\colon\{0,1\}^n\to\su[m]$.
\end{repeatresult}

We will usually describe a formal cube in terms of its \emph{label functions}
$f_i\colon \{0,1\}^n\to\{0,1\}$ for $0\le i\le m$ which are given by
\[
    i\in X(x)\quad \iff \quad f_i(x)=1.
\]
Among those formal cubes the most important class is the one we call \emph{degenerable}, allowing us to translate the degeneracies of the simplicial set to the degeneracies of the cubical set.

\begin{repeatresult}{definition}{def:formal-degenerable}
A formal $n$-cube $X$ with labels $[n]$ is called \emph{degenerable}
if for every $0\le i<n$, identifying labels $i$ and $i+1$
yields a cube which is degenerate along the $i$-th coordinate direction.
\end{repeatresult}

In terms of label functions this means that for every $0\le i<n$, the pointwise supremum
$f_i\vee f_{i+1}$ does not depend on $x_i$.
This allows us to fully classify degenerable cubes, see Section \ref{ssec:classification}.
With these definitions we have everything at hand to define the kernel $K$.

\begin{repeatresult}{definition}{def:kernel}
Define the functor $K\colon{\bbox}^\op\times\bbDelta\to\set$ 
as the subfunctor of the functor
\begin{center}
\begin{tikzcd}
	{{\bbox}^\op\times\bbDelta} & {\set^\op\times\set} & \set
	\arrow["{?\times\su}", from=1-1, to=1-2]
	\arrow["\hom", from=1-2, to=1-3]
\end{tikzcd}
\end{center}
generated by the assignment
\[
(\langle m\rangle,[n])\mapsto
\begin{cases}
    \{A\colon\{0,1\}^n\to\su[n]\mid A\;\text{degenerable}\},\quad &\text{if } m = n, \\
    \emptyset,\quad &\text{if } m\neq n.
\end{cases}
\]
\end{repeatresult}

Pullback along $K$ yields the functor from cubesets to simplicial sets.

\begin{repeatresult}{definition}{def:functor}
Define the functor $K^\ast\colon\cSet\to\sSet$ by
\[
    \cSet\ni X\mapsto K^\ast X \coloneqq \hom_{\cSet}(K_{(-)}, X)\in\sSet.
\]
\end{repeatresult}

In order to prove the main theorem, one step is to explicitly describe the value of the functor on $X$.
To do so, the necessary combinatorics can be understood via
our main technique of gluing formal cubes along common faces
to obtain a new formal cube that has the other two faces as its faces.

Combinatorially, we need to translate the simplicial horn filling property to the cubical corner completion,
and at the same time translate the equivalence relation on spheres of the simplicial homotopy groups to the cubical equivalence relation for the structure groups, using corner completions again.
The double usage of the cubical corner as horn and as boundary, see Section \ref{ssec:differences},
forces us to have two prominent $n$-cubes in $K$: One, $D_n$, for the \enquote{spatial direction/horns}, and another, $C_n$, for the \enquote{homotopical direction/boundary}.
However, since in the cubical theory all information is equally stored on the level of points, these cubes have to interfere appropriately.

The correct choice turns out to be the following, having the neat universal property that $C_n$ is (up to flips) the unique cube such that every $n$-tuple
appears exactly once and the full label doesn't appear, see Proposition \ref{prop:unique-Cn}.
\begin{repeatresult}{example}{ex:Cn-Dn}
Define formal cubes $C_n$ by the label functions
\[
    f_i\colon\{0,1\}^n\to\{0,1\},\quad x\mapsto x_i\vee\neg x_{i-1}
\]
with the convention that $x_{-1}=1$ and $x_n=0$.

Define formal cubes $D_n$ with label functions given by the ones of $C_n$ for $i<n$ and
$f_n\equiv 1$.
\end{repeatresult}

Now, these cubes $C_n$ and $D_n$ behave precisely as one wants them to.

\begin{repeatresult}{theorem}{thm:C-D-glue-all-really}
Every degenerable $n$-cube is generated, under gluing and cube morphisms,
from the cubes $i_\ast C_k$ and $i_\ast D_k$ where $i\colon[k]\to[n]$ runs over all injections in $\bbDelta$.
\end{repeatresult}
This allows us to describe the value of $K^{*}X$ purely in terms of certain $C_n$ and $D_n$ configurations, thereby reducing to the study of $C_n$ and $D_n$.
\begin{repeatresult}{proposition}{prop:concrete-simplices}
Let $X$ be a concrete cubeset with gluing property.
Then $K^\ast X_n$ consists of all maps $f\colon\su[n]\to X(0)$ such that for every injection
$i\colon[m]\to [n]$ in $\bbDelta$ the $m$-cubes
\begin{align*}
    f\circ i_\ast\circ C_m&\colon\{0,1\}^m\to\su[m]\to\su[n]\to X(0) \\
    f\circ i_\ast\circ D_m&\colon\{0,1\}^m\to\su[m]\to\su[n]\to X(0)
\end{align*}
are in $X(m)$.
\end{repeatresult}
Having the description of the value of the functor in terms of the $C_n$ and $D_n$ cubes,
we can proceed to show that they translate the desired structures.
For this, it is crucial to check that for fibrant $X$, every ($n+1$)-horn of $K^{*}X$ admits a filler.
To do so, we show that given the corresponding $n$-simplices appearing in the horn,
they can be assembled to the corner of the cube $C_{n+1}$.
Corner completion in $X$ to $C_{n+1}$ tells us that the simplicial boundary exists in $K^{*}X$.
Using the simplicial boundary to construct the corner of the cube $D_{n+1}$ we now obtain the actual filler of the simplicial horn.
The following theorem is the corresponding central assertion on the side of formal cubes, which we prove in Section \ref{ssec:generating_faces}.

\begin{repeatresult}{theorem}{thm:faces_of_Cn_generated}
The faces of $C_{n+1}$ are generated under gluing by $C_n^i$ and $D_n^i$ for $i\le n+1$.
\end{repeatresult}

\subsubsection*{The simplicial set associated to a concrete abelian cubegroup}

As an explicit example, we compute the simplicial set $K^*\HK(A_\bullet)$ associated to a concrete abelian cubegroup.

\begin{proposition*}
Consider the Host-Kra cubegroup $\HK(A_\bullet)$ associated to a filtered abelian group $A_\bullet$.
Then the simplicial abelian group $K^*\HK(A_\bullet)$ is isomorphic,
under Dold-Kan correspondence, to the connective chain complex
\begin{center}
\begin{tikzcd}
	\cdots & {A_3^2} && {A_2^2} && {A_1^2} && {A_0}
	\arrow[from=1-1, to=1-2]
	\arrow["{(a,b)\mapsto(b,0)}", from=1-2, to=1-4]
	\arrow["{(a,b)\mapsto (b,0)}", from=1-4, to=1-6]
	\arrow["{(a,b)\mapsto b}", from=1-6, to=1-8]
\end{tikzcd}.
\end{center}
\end{proposition*}

\begin{proof}
We have to show that the described complex is the reduced Moore complex of $K^*\HK(A_\bullet)$.
But $\bigcap_{i\le n-1}\ker d_i$ consists of precisely those $x\colon \su[n] \to A_0$ in $K^*\HK(A_\bullet)$ whose values are $0$ everywhere except possibly on
$[n]$ and $[n-1]$.
For such elements, the condition on the $C_n$ cube being in $\HK(A_\bullet)$ reduces to the $[n-1]$-coordinate being in $A_n$,
and the $D_n$ cube precisely yields that the $[n]$-coordinate is in $A_n$.
Thus the chain groups are $A_n^2$.
The differential $d_n$ precisely singles out the $[n-1]$-coordinate, resulting in the desired Moore complex.
\end{proof}

\begin{corollary*}
The simplicial abelian group $K^\ast\cD_n(A)$ is isomorphic to
\[
    \bigoplus_{k=0}^{n-1}E(A,k)\oplus K(A,n),
\]
    where $K(A,n)$ is the $n$-th Eilenberg-Mac Lane space of $A$
and $E(A,k)$ is contractible.
\end{corollary*}

We refer to the upcoming work \cite{Bihlmaier2026b} for a detailed treatment, especially in the non-commutative case,
exhibiting deep connections to (non-abelian) Dold-Kan correspondence.

\subsection{Nilspace Theory vs. Simplicial Homotopy Theory}\label{ssec:differences}

Let us briefly explain some similarities between the theory of nilspaces and (simplicial) homotopy theory which have motivated our results.
Nilspace theory is a structure theory of cubical sets, being central to ergodic structure theory and higher order Fourier analysis, and we refer to \cite{Bihlmaier2026} as well as \cite{Candela2017,Candela2017a,Gutman2020} for detailed introductions.
Simplicial sets and their homotopy theory in turn are central in large parts of modern algebra, as, e.g. in algebraic topology, higher category theory, and homological algebra.
We refer to  \cite{Weibel1994,Goerss1999,Lurie2009,Kerodon2026}.

The similarity of nilspace theory with (a cubical variant of) simplicial homotopy theory starts with the very definition of cubesets:
Cubesets are built out of various cubes of different dimensions,
whilst simplicial sets are glued from multiple simplices.
This datum is encoded in a presheaf, telling us how many $n$-dimensional shapes are attached in which way to other $m$-dimensional shapes.
\begin{definition}\leavevmode
    \begin{enumerate}
        \item \begin{enumerate}[(i)]
                  \item
The \emph{category of cubes $\bbox$} is the category with objects ${\bbox}^n$ for $n\in\N_0$
and morphisms $f\colon{\bbox}^n\to{\bbox}^m$ given by affine linear maps $\Z^n\to\Z^m$ that map $\{0,1\}^n$ into $\{0,1\}^m$.

                  \item
The category $\cSet$ of cubesets is the presheaf category $\PSh(\bbox)$.

\end{enumerate}
        \item \begin{enumerate}[(i)]
                  \item
The \emph{simplex category $\bbDelta$} is the category of non-empty finite ordinals and non-decreasing maps.
                  \item
The category $\sSet$ of simplicial sets
is the presheaf category $\PSh(\bbDelta)$.
\end{enumerate}
    \end{enumerate}
\end{definition}

The low-dimensional shapes can be pictured as follows.

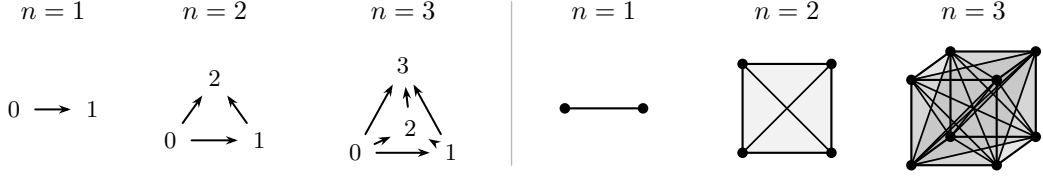
\pgfdeclarelayer{combinedsymfacefill}
\pgfdeclarelayer{combinedsymfaceoutline}
\pgfsetlayers{combinedsymfacefill,combinedsymfaceoutline,main}

\newcommand{\CombinedSymFace}[1]{%
  \begin{pgfonlayer}{combinedsymfacefill}
    \path[fill=gray, fill opacity=0.10] #1 -- cycle;
  \end{pgfonlayer}
  \begin{pgfonlayer}{combinedsymfaceoutline}
    \path[draw=gray!60,line width=0.45pt] #1 -- cycle;
  \end{pgfonlayer}
}

\newcommand{\CombinedSymDiagFace}[1]{%
  \begin{pgfonlayer}{combinedsymfacefill}
    \path[fill=gray, fill opacity=0.16] #1 -- cycle;
  \end{pgfonlayer}
  \begin{pgfonlayer}{combinedsymfaceoutline}
    \path[draw=gray!70,line width=0.5pt] #1 -- cycle;
  \end{pgfonlayer}
}

\begin{center}
\begin{tikzpicture}[
    x=1cm,y=1cm,
    font=\small,
    line cap=round,
    line join=round,
    vertex/.style={circle,fill,inner sep=1.4pt},
    mainedge/.style={draw=black, line width=0.8pt},
    affdiag/.style={draw=black, line width=0.65pt},
    simpedge/.style={
        draw=black,
        line width=0.75pt,
        -{Stealth[length=1.6mm,width=1.1mm]},
        shorten <= 8.0pt,
        shorten >= 8.0pt
    },
    collab/.style={font=\normalsize},
    scale=0.47
]

\node[collab] at ( 1.10, 4.0) {$n=1$};
\node[collab] at ( 5.65, 4.0) {$n=2$};
\node[collab] at (10.95, 4.0) {$n=3$};

\begin{scope}[shift={(0,1.22)}]
    \node at (0,0) {$0$};
    \node at (2.2,0) {$1$};
    \draw[simpedge] (0,0) -- (2.2,0);
\end{scope}

\begin{scope}[shift={(4.4,0.35)}]
    \coordinate (sA) at (0,0);
    \coordinate (sB) at (2.5,0);
    \coordinate (sC) at (1.25,1.75);

    \draw[simpedge] (sA)--(sB);
    \draw[simpedge] (sA)--(sC);
    \draw[simpedge] (sB)--(sC);

    \node at (sA) {$0$};
    \node at (sB) {$1$};
    \node at (sC) {$2$};
\end{scope}

\begin{scope}[shift={(9.6,0)}]
    \coordinate (tA) at (0,0);
    \coordinate (tB) at (2.7,0);
    \coordinate (tC) at (1.55,0.7);
    \coordinate (tD) at (1.35,2.45);

    \draw[simpedge] (tA)--(tB);
    \draw[simpedge] (tA)--(tC);
    \draw[simpedge] (tB)--(tC);
    \draw[simpedge] (tA)--(tD);
    \draw[simpedge] (tB)--(tD);
    \draw[simpedge] (tC)--(tD);

    \node at (tA) {$0$};
    \node at (tB) {$1$};
    \node at (tC) {$2$};
    \node at (tD) {$3$};
\end{scope}

\draw[gray!55,line width=0.5pt] (14.0,-0.35) -- (14.0,4.25);

\node[collab] at (16.60,4.0) {$n=1$};
\node[collab] at (21.75,4.0) {$n=2$};
\node[collab] at (27.00,4.0) {$n=3$};

\begin{scope}[shift={(15.5,1.25)}]
    \coordinate (cAone) at (0,0);
    \coordinate (cBone) at (2.2,0);

    \draw[mainedge] (cAone)--(cBone);
    \node[vertex] at (cAone) {};
    \node[vertex] at (cBone) {};
\end{scope}

\begin{scope}[shift={(20.5,0)}]
    \coordinate (cAtwo) at (0,0);
    \coordinate (cBtwo) at (2.5,0);
    \coordinate (cCtwo) at (0,2.5);
    \coordinate (cDtwo) at (2.5,2.5);

    \CombinedSymFace{(cAtwo)--(cBtwo)--(cDtwo)--(cCtwo)}

    \draw[mainedge] (cAtwo)--(cBtwo);
    \draw[mainedge] (cBtwo)--(cDtwo);
    \draw[mainedge] (cDtwo)--(cCtwo);
    \draw[mainedge] (cCtwo)--(cAtwo);

    \draw[affdiag] (cAtwo)--(cDtwo);
    \draw[affdiag] (cBtwo)--(cCtwo);

    \foreach \P in {cAtwo,cBtwo,cCtwo,cDtwo} {
        \node[vertex] at (\P) {};
    }
\end{scope}

\begin{scope}[shift={(25.25,-0.35)}]
    \coordinate (cA) at (0,0);
    \coordinate (cB) at (2.4,0);
    \coordinate (cC) at (0,2.4);
    \coordinate (cD) at (2.4,2.4);

    \coordinate (cE) at (1.1,0.8);
    \coordinate (cF) at (3.5,0.8);
    \coordinate (cG) at (1.1,3.2);
    \coordinate (cH) at (3.5,3.2);

    \CombinedSymFace{(cA)--(cB)--(cD)--(cC)}
    \CombinedSymFace{(cE)--(cF)--(cH)--(cG)}
    \CombinedSymFace{(cA)--(cB)--(cF)--(cE)}
    \CombinedSymFace{(cC)--(cD)--(cH)--(cG)}
    \CombinedSymFace{(cA)--(cC)--(cG)--(cE)}
    \CombinedSymFace{(cB)--(cD)--(cH)--(cF)}

    \CombinedSymDiagFace{(cA)--(cB)--(cH)--(cG)}
    \CombinedSymDiagFace{(cA)--(cC)--(cH)--(cF)}
    \CombinedSymDiagFace{(cA)--(cE)--(cH)--(cD)}

    \draw[mainedge] (cA)--(cB);
    \draw[mainedge] (cB)--(cD);
    \draw[mainedge] (cD)--(cC);
    \draw[mainedge] (cC)--(cA);

    \draw[mainedge] (cE)--(cF);
    \draw[mainedge] (cF)--(cH);
    \draw[mainedge] (cH)--(cG);
    \draw[mainedge] (cG)--(cE);

    \draw[mainedge] (cA)--(cE);
    \draw[mainedge] (cB)--(cF);
    \draw[mainedge] (cC)--(cG);
    \draw[mainedge] (cD)--(cH);

    \draw[affdiag] (cA)--(cD);
    \draw[affdiag] (cB)--(cC);
    \draw[affdiag] (cE)--(cH);
    \draw[affdiag] (cF)--(cG);
    \draw[affdiag] (cA)--(cF);
    \draw[affdiag] (cB)--(cE);
    \draw[affdiag] (cC)--(cH);
    \draw[affdiag] (cD)--(cG);
    \draw[affdiag] (cA)--(cG);
    \draw[affdiag] (cC)--(cE);
    \draw[affdiag] (cB)--(cH);
    \draw[affdiag] (cD)--(cF);

    \draw[affdiag] (cA)--(cH);
    \draw[affdiag] (cB)--(cG);
    \draw[affdiag] (cC)--(cF);
    \draw[affdiag] (cD)--(cE);

    \foreach \P in {cA,cB,cC,cD,cE,cF,cG,cH} {
        \node[vertex] at (\P) {};
    }
\end{scope}

\end{tikzpicture}
\end{center}

Now, in every dimension $n>0$, there is a \emph{horn} resp. \emph{corner},
which is supposed to be an $n$-dimensional shape which roughly has one \enquote{piece} of information missing.
In the simplicial theory, the information is stored as labels on the faces resulting in a \emph{horn},
while in the cubical theory, one stores information usually only on \emph{point-level},%
\footnote{These cubesets are called \emph{concrete}, i.e. $X(n)\hookrightarrow X(0)^{2^n}$, see \cite[Definition 2.3.1]{Bihlmaier2026}.
Usually, one directly restricts everything to the full subcategory of concrete cubesets.}
resulting in a \emph{corner} (the cube without one vertex).
This allows for the definition of fibrant objects and fibrations.

\begin{definition}\leavevmode
\begin{enumerate}
        \item
\begin{enumerate}[(i)]

    \item
For $n\in\N$, the \emph{$n$-corner ${\bbcor}^n$} is the union of all linear faces of ${\bbox}^n$
of dimension less than $n$ that don't contain the point $(1,\ldots,1)$.

    \item
A cubeset $X$ is called \emph{fibrant} if it has \emph{corner completion},
i.e., for every corner ${\bbcor}^n\to X$ in $X$ there exists a completion to a cube ${\bbox}^n\to X$
along the inclusion ${\bbcor}^n\to{\bbox}^n$.%
\footnote{Classically, concrete cubespaces with corner completion are also referred to as \emph{nilspaces}.}

    \item
A map $f\colon X\to Y$ between cubesets is called a \emph{fibration}
if for every commutative diagram
\begin{center}
\begin{tikzcd}[ampersand replacement=\&]
	{{\bbcor}^n} \& X \\
	{{\bbox}^n} \& Y
	\arrow[from=1-1, to=1-2]
	\arrow[hook, from=1-1, to=2-1]
	\arrow[from=1-2, to=2-2]
	\arrow[dashed, from=2-1, to=1-2]
	\arrow[from=2-1, to=2-2]
\end{tikzcd}
\end{center}
there exists a diagonal lift ${\bbox}^n\to X$.
\end{enumerate}
        \item
\begin{enumerate}[(i)]
    \item For $n\in\N$ and $i\in[n]$, the \emph{$i$-th $n$-horn $\Lambda^n_i$} is the union of all faces of $\bbDelta^n$
of dimension less than $n$ that contain $i$.
    So $\Lambda^n_i$ is obtained from $\bbDelta^n$ by deleting the ($n-1$)-dimensional face opposite to $i$.

\item A simplicial set $X$ is called a \emph{Kan complex} if it has \emph{horn filling}:
every $n$-horn $\Lambda^n_i$ in $X$ can be extended to an $n$-simplex in $X$.%
\footnote{These are the fibrant objects in the Kan-Quillen model structure on simplicial sets.}
    \item A morphism $f\colon X\to Y$ between simplicial sets is called a (Kan-)\emph{fibration}
if for every commutative diagram
\begin{center}
\begin{tikzcd}[ampersand replacement=\&]
	{{\Lambda}_i^n} \& X \\
	{{\bbDelta}^n} \& Y
	\arrow[from=1-1, to=1-2]
	\arrow[hook, from=1-1, to=2-1]
	\arrow[from=1-2, to=2-2]
	\arrow[dashed, from=2-1, to=1-2]
	\arrow[from=2-1, to=2-2]
\end{tikzcd}
\end{center}
there exists a lift ${\bbDelta}^n\to X$.
\end{enumerate}
\end{enumerate}
\end{definition}

The first $n$-horns and $n$-corners may be visualized as follows.
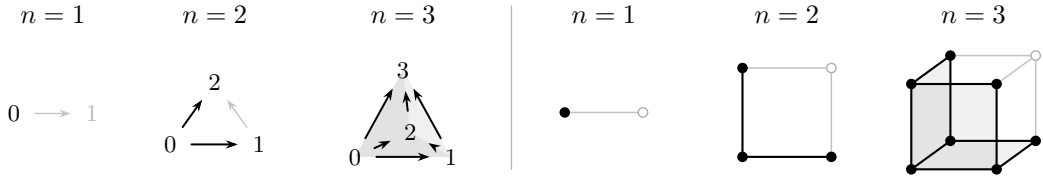
\begin{center}
\begin{tikzpicture}[
    x=1cm,y=1cm,
    font=\small,
    line cap=round,
    line join=round,
    vertex/.style={
        circle,
        fill,
        inner sep=1.4pt
    },
    missingvertex/.style={
        circle,
        draw=gray!60,
        fill=white,
        line width=0.55pt,
        inner sep=1.25pt
    },
    mainedge/.style={
        draw=black,
        line width=0.8pt
    },
    ghostedge/.style={
        draw=gray!48,
        line width=0.55pt
    },
    simpedge/.style={
        draw=black,
        line width=0.75pt,
        -{Stealth[length=1.6mm,width=1.1mm]},
        shorten <= 8.0pt,
        shorten >= 8.0pt
    },
    ghostsimpedge/.style={
        draw=gray!45,
        line width=0.55pt,
        -{Stealth[length=1.5mm,width=1.0mm]},
        shorten <= 8.0pt,
        shorten >= 8.0pt
    },
    hornface/.style={
        fill=gray,
        fill opacity=0.12,
        draw=none
    },
    hornedgebg/.style={
        draw=gray,
        opacity=0.22,
        line width=3.0pt
    },
    cornerface/.style={
        fill=gray,
        fill opacity=0.11,
        draw=none
    },
    collab/.style={
        font=\normalsize
    },
    scale=0.47
]

\node[collab] at ( 1.10, 4.0) {$n=1$};
\node[collab] at ( 5.65, 4.0) {$n=2$};
\node[collab] at (10.95, 4.0) {$n=3$};

\begin{scope}[shift={(0,1.22)}]
    \coordinate (hAone) at (0,0);
    \coordinate (hBone) at (2.2,0);

    \draw[ghostsimpedge] (hAone)--(hBone);

    \node at (hAone) {$0$};
    \node[text=gray!55] at (hBone) {$1$};
\end{scope}

\begin{scope}[shift={(4.4,0.35)}]
    \coordinate (hAtwo) at (0,0);
    \coordinate (hBtwo) at (2.5,0);
    \coordinate (hCtwo) at (1.25,1.75);

    \draw[simpedge] (hAtwo)--(hBtwo);
    \draw[simpedge] (hAtwo)--(hCtwo);

    \draw[ghostsimpedge] (hBtwo)--(hCtwo);

    \node at (hAtwo) {$0$};
    \node at (hBtwo) {$1$};
    \node at (hCtwo) {$2$};
\end{scope}

\begin{scope}[shift={(9.6,0)}]
    \coordinate (hA) at (0,0);
    \coordinate (hB) at (2.7,0);
    \coordinate (hC) at (1.55,0.7);
    \coordinate (hD) at (1.35,2.45);

    \path[hornface] (hA)--(hB)--(hC)--cycle;
    \path[hornface] (hA)--(hB)--(hD)--cycle;
    \path[hornface] (hA)--(hC)--(hD)--cycle;

    \draw[simpedge] (hA)--(hB);
    \draw[simpedge] (hA)--(hC);
    \draw[simpedge] (hB)--(hC);
    \draw[simpedge] (hA)--(hD);
    \draw[simpedge] (hB)--(hD);
    \draw[simpedge] (hC)--(hD);

    \node at (hA) {$0$};
    \node at (hB) {$1$};
    \node at (hC) {$2$};
    \node at (hD) {$3$};
\end{scope}

\draw[gray!55,line width=0.5pt]
    (14.0,-0.35) -- (14.0,4.25);

\node[collab] at (16.60,4.0) {$n=1$};
\node[collab] at (21.75,4.0) {$n=2$};
\node[collab] at (27.00,4.0) {$n=3$};

\begin{scope}[shift={(15.5,1.25)}]
    \coordinate (cAone) at (0,0);
    \coordinate (cBone) at (2.2,0);

    \draw[ghostedge] (cAone)--(cBone);

    \node[vertex]        at (cAone) {};
    \node[missingvertex] at (cBone) {};
\end{scope}

\begin{scope}[shift={(20.5,0)}]
    \coordinate (cAtwo) at (0,0);
    \coordinate (cBtwo) at (2.5,0);
    \coordinate (cCtwo) at (0,2.5);
    \coordinate (cDtwo) at (2.5,2.5);

    \draw[mainedge] (cAtwo)--(cBtwo);
    \draw[mainedge] (cAtwo)--(cCtwo);

    \draw[ghostedge] (cBtwo)--(cDtwo);
    \draw[ghostedge] (cCtwo)--(cDtwo);

    \node[vertex]        at (cAtwo) {};
    \node[vertex]        at (cBtwo) {};
    \node[vertex]        at (cCtwo) {};
    \node[missingvertex] at (cDtwo) {};
\end{scope}

\begin{scope}[shift={(25.25,-0.35)}]
    \coordinate (cA) at (0,0);
    \coordinate (cB) at (2.4,0);
    \coordinate (cC) at (0,2.4);
    \coordinate (cD) at (2.4,2.4);

    \coordinate (cE) at (1.1,0.8);
    \coordinate (cF) at (3.5,0.8);
    \coordinate (cG) at (1.1,3.2);
    \coordinate (cH) at (3.5,3.2);

    \path[cornerface] (cA)--(cB)--(cD)--(cC)--cycle;
    \path[cornerface] (cA)--(cB)--(cF)--(cE)--cycle;
    \path[cornerface] (cA)--(cC)--(cG)--(cE)--cycle;

    \draw[mainedge] (cA)--(cB);
    \draw[mainedge] (cB)--(cD);
    \draw[mainedge] (cD)--(cC);
    \draw[mainedge] (cC)--(cA);

    \draw[mainedge] (cA)--(cE);
    \draw[mainedge] (cB)--(cF);
    \draw[mainedge] (cE)--(cF);

    \draw[mainedge] (cC)--(cG);
    \draw[mainedge] (cE)--(cG);

    \draw[ghostedge] (cD)--(cH);
    \draw[ghostedge] (cF)--(cH);
    \draw[ghostedge] (cG)--(cH);

    \foreach \P in {cA,cB,cC,cD,cE,cF,cG} {
        \node[vertex] at (\P) {};
    }
    \node[missingvertex] at (cH) {};
\end{scope}

\end{tikzpicture}
\end{center}
The fibrant objects form a subcategory of well-behaved objects in the sense that, e.g., strict limits of fibrant diagrams behave as one homotopically would like them to do.
Also, the most important examples that come up (such as group objects) are fibrant, and one has a fibrant replacement available.

If this lifting is unique from some dimension $k$ on, this allows us to reduce to this dimension,
resulting in a (strict) notion of a $k$-truncated object.

\begin{definition}\leavevmode
    \begin{enumerate}
        \item
\begin{enumerate}[(i)]
    \item A cubeset $X$ is called \emph{$k$-step} if for all $n\geq k+1$ the map
\[
    \hom({\bbox}^n, X) \to \hom({\bbcor}^n, X)
\]
is bijective, i.e., each $n$-corner in $X$ has a unique completion to an $n$-cube of $X$.
        \item The left adjoint to the inclusion of $k$-step cubesets is called the $k$-\emph{truncation} $\tau_k$.
    \end{enumerate}
        \item
        \begin{enumerate}[(i)]
            \item
A Kan complex $X$ is called a \emph{$k$-groupoid} if for all $n\geq k+1$ the map
\[
    \hom(\bbDelta^n,X)\to\hom(\Lambda^n,X)
\]
is bijective, i.e., every $n$-horn in $X$ has a unique completion to an $n$-simplex of $X$.

        \item The left adjoint to the inclusion of $k$-groupoids into Kan complexes is called the {\emph{fundamental $k$-groupoid}} $\pi_{\le k}$.
        \end{enumerate}
        \end{enumerate}
    \end{definition}

For good fibrant objects, there is a simple equivalence relation to compute the truncations by some equivalence relation (see e.g. \cite[Construction 054D]{Kerodon2026}).
In general, however, the truncation involves some iterated identification (and gluing) steps, since identifying two simplices whose horns agree identifies lower dimensional data, and similarly with the cubical situation, where even points get identified.

One important feature of the truncations is that they allow us to inductively reconstruct the space from its lower dimensional information,
in each step with some extension procedure by an (abelian) group, the \emph{structure}, respectively \emph{homotopy group}.
These homotopy groups somewhat are an algebraic structure on the fibres of the truncation:
The fiber of $\tau_n X\to \tau_{n-1} X$ over some point $x_0$ has to be thought of as the set of all $n$-dimensional spheres, equipped with the following multiplication.

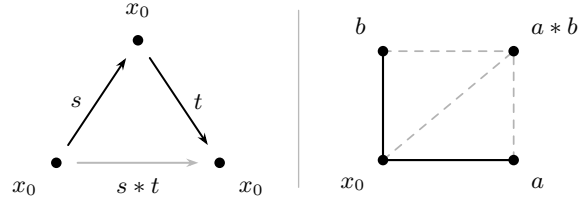
\begin{center}
\begin{tikzpicture}[
    x=1cm,y=1cm,
    font=\small,
    line cap=round,
    line join=round,
    vertex/.style={
        circle,
        fill,
        inner sep=1.4pt
    },
    simpedge/.style={
        draw=black,
        line width=0.75pt,
        -{Stealth[length=1.6mm,width=1.1mm]},
        shorten <= 8.0pt,
        shorten >= 8.0pt
    },
    cubedge/.style={
        draw=black,
        line width=0.8pt
    },
    dashededge/.style={
        draw=gray!60,
        line width=0.65pt,
        dashed
    },
    scale=0.72
]

\begin{scope}[shift={(0,0)}]
    \coordinate (A) at (0,0);
    \coordinate (B) at (3.0,0);
    \coordinate (C) at (1.5,2.25);

    \draw[simpedge]
        (A)--node[midway,left=2pt] {$s$}(C);

    \draw[simpedge]
        (C)--node[midway,right=2pt] {$t$}(B);

    \draw[simpedge,draw=gray!55]
        (A)--node[midway,below=2pt,text=black] {$s*t$}(B);

    \node[
        vertex,
        label={[label distance=3pt]225:$x_0$}
    ] at (A) {};

    \node[
        vertex,
        label={[label distance=3pt]315:$x_0$}
    ] at (B) {};

    \node[
        vertex,
        label={[label distance=3pt]90:$x_0$}
    ] at (C) {};
\end{scope}

\draw[gray!55,line width=0.5pt]
    (4.45,-0.45) -- (4.45,2.80);

\begin{scope}[shift={(6.0,0.05)}]
    \coordinate (X)  at (0,0);
    \coordinate (A2) at (2.4,0);
    \coordinate (B2) at (0,2.0);
    \coordinate (AB) at (2.4,2.0);

    \draw[cubedge] (X)--(A2);
    \draw[cubedge] (X)--(B2);

    \draw[dashededge] (B2)--(AB);
    \draw[dashededge] (A2)--(AB);
    \draw[dashededge] (X)--(AB);

    \node[vertex] at (X)  {};
    \node[vertex] at (A2) {};
    \node[vertex] at (B2) {};
    \node[vertex] at (AB) {};

    \node[below left=3pt]  at (X)  {$x_0$};
    \node[below right=3pt] at (A2) {$a$};
    \node[above left=3pt]  at (B2) {$b$};
    \node[above right=3pt] at (AB) {$a\ast b$};
\end{scope}

\end{tikzpicture}
\end{center}
Note that for the simplicial picture, for composition to make sense we have to insist that the codomain of $s$ is the same as the domain of $t$, leading to the necessity to restrict to spheres instead of simply requiring them to agree on the $n$-horn, in contrast to the cubical world.
This leads to all sorts of different towers and truncations as e.g. the (weak) coskeleton tower.

In general, this doesn't give a well-defined composition since there are possibly many different horn fillings in $X$.
Making this operation well-defined precisely amounts to passing to the $n$-truncation.

\begin{definition}\leavevmode
    \begin{enumerate}
        \item  Let $X$ be a concrete fibrant cubeset and $x_0$ a point in $X$.
        Define the \emph{$k$-th structure group} $\pi_k(X)=\pi_k(X,x_0)$ to be the set of $k$-cubes in $\tau_k X$ whose restriction to the $k$-corner $\bbcor^k$ is constant $x_0$.
        Equip $\pi_k(X)$ with the binary operation induced by corner completion as depicted above.
This makes $\pi_k(X)$ an abelian group for $k\ge 1$.

        \item  Let $X$ be a Kan complex and $x_0$ a point in $X$.
        Define the \emph{$k$-th homotopy group} $\pi_k(X)=\pi_k(X,x_0)$ to be the set of $k$-simplices in $\pi_{\le k} X$ whose restriction to the $k$-\emph{boundary} $\partial \bbDelta^k$ is constant $x_0$.
        Equip $\pi_k(X)$ with a binary operation induced by horn filling as depicted above.
This makes $\pi_k(X)$ a group for $k\ge 1$, which is abelian for $k\ge 2$.
   \end{enumerate}
\end{definition}

Luckily, for fibrant objects, the computation of the truncation localizes onto spheres and
reduces to a simple identification step there, essentially making the composition with the identity well-defined, which suffices for constructing a group.
\begin{proposition}\leavevmode
    \begin{enumerate}
        \item     Let $X$ be a concrete fibrant cubeset and $x_0$ a point in $X$.
        Then $\pi_n(X)$ is equivalently the set of $n$-cubes of $X$ with corner $x_0$ modulo the equivalence relation given by identifying two such $n$-cubes $s,t$
        if there exists an $n+1$-cube in $X$
        whose restriction to the last lower face is $s$, a diagonal is $t$ and the rest of the points are $x_0$.

        \item
    Let $X$ be a Kan complex and $x_0$ a point in $X$.
        Then $\pi_n(X)$ can be computed as the set of $n$-spheres of $X$ based at $x_0$ modulo the following equivalence relation:
Two $n$-spheres $s,t$ in $X$ based at $x_0$ are equivalent if there exists an $n+1$-simplex with last face $s$, second-to-last face $t$, and the rest of the faces all being constant $x_0$.
\end{enumerate}
    \end{proposition}

We now see that this group structure is compatible with the truncations.
In fact, one can reverse-engineer the notions of truncation etc., starting with the desire to make this composition well-defined.
More precisely, in both situations, iteratively applying the truncations yields a tower of factors
\begin{center}
\begin{tikzcd}
	X & \cdots & {\tau_{2}X} & {\tau_1 X} & {\tau_0 X} \\
	X & \cdots & {\pi_{\le 2 } X} & {\pi_{\le 1} X} & {\pi_{\le 0} X}
	\arrow[from=1-1, to=1-2]
	\arrow[from=1-2, to=1-3]
	\arrow[from=1-3, to=1-4]
	\arrow[from=1-4, to=1-5]
	\arrow[from=2-1, to=2-2]
	\arrow[from=2-2, to=2-3]
	\arrow[from=2-3, to=2-4]
	\arrow[from=2-4, to=2-5]
\end{tikzcd}
\end{center}
Note that in both cases the transition maps are fibrations, and the maps $X\to \tau_n X$ resp. $X\to \pi_{\le n}X$ induce isomorphisms on all $\pi_i$ for $i\le n$ and fulfill $\pi_i(\tau_n X)=0$ resp. $\pi_i(\pi_{\le n}X)=0$ for $i>n$.
Towers with the latter property are classically called \emph{Postnikov towers}.
This particular tower is called the \emph{canonical Postnikov tower} in the case of simplicial sets,
and the \emph{Host-Kra structure tower} in the case of cubesets.

These statements, usually with a concrete description of the fibers as Eilenberg-Mac Lane spaces and a bundle formulation,
are called the \emph{weak structure theorem} in the case of cubical sets.

There are many more similarities between both theories, e.g., connectedness of simplicial sets corresponds to the notion of ergodicity in the cubical terminology.
Let us remark on some further similarities, on which we plan to elaborate in a future work in more detail \cite{Bihlmaier2026b}.

\begin{remark}
    \begin{itemize}
        \item
There are natural Dold-Kan style statements in both settings.
        In the cubical setting, one has an equivalence between filtered groups and concrete cubegroups.
        In the simplicial setting, it is an equivalence between hypercrossed complexes of groups and simplicial groups,
        specializing in the abelian case to an equivalence between connective chain complexes of abelian groups and simplicial abelian groups.
         \item
    Using Dold-Kan, one arrives at canonical definitions of Eilenberg-Mac Lane spaces in both settings.
    This in turn opens the question to ask for the fiber of the transition map in the Postnikov tower to be isomorphic to $K(A,n)$ rather than weakly equivalent.
    In the cubical setting this isomorphism is strict on the nose.
  In the simplicial setting, this is achieved by passing to a twisted Postnikov tower
(with the trade-off of losing strictness in $X=\varprojlim \tau_{n}X$), see~\cite{May1967}.
        This allows us to strengthen the Postnikov statements to $\tau_k X\to\tau_{k-1} X$ being a $D_k(A)$-fiber bundle,
        resp. $\pi_{\le k}X\to \pi_{\le k-1}X$ being a $K(\pi_k(X),k)$-bundle.

    \item  In both theories, there exist definitions of cocycles, coboundaries, cohomology groups and extensions,
        and the usual relationships hold.
    \end{itemize}
\end{remark}

To conclude our discussion of similarities between the theory of nilspaces and simplicial homotopy theory let us emphasize one major difference.
For cubesets, the $n$-corner served both as the horn \emph{and} the boundary,
whereas for simplicial sets there is a distinction.
For cubesets, one mainly considers concrete presheaves, so every construction has an impact on the underlying set.
Thus, one cannot create non-trivial boundaries if the boundary does not \enquote{miss a point}.
This is in contrast to simplicial sets where the homotopical information is usually not stored on the underlying points but rather written on the \emph{labels} of the higher dimensional faces themselves.
This is precisely what makes the translation from cubesets to simplicial sets difficult and is solved combinatorially by introducing the $C_n$ and $D_n$ cubes.

\subsubsection*{Acknowledgments}

We thank Asgar Jamneshan for many valuable discussions related to this paper and for comments on an earlier draft.

\subsubsection*{Disclosure of AI Usage}

We have made use of ChatGPT to discuss some of the ideas of the paper which have been helpful to gain insight and intuition.
It has provided Python code that computes configurations of degenerable cubes
and the gluing closures of the cubes $C_n$ and $D_n$,
and an HTML script that allows us to interactively display the generated cubes
and navigate through the data set.
Moreover, it suggested to use label functions in Definition~\ref{def:formal-cubes}
and observed uniqueness of the cube $C_n$ for which it provided the first proof of Proposition~\ref{prop:unique-Cn}, which has then been improved upon by the authors.
Lastly, many of the included TikZ-figures are written by ChatGPT.
We also used it to correct spelling mistakes and check proofs.

Other than in the above disclosed parts, there has been no involvement of artificial intelligence in this paper.

\section{Combinatorial Considerations}

In this chapter we develop the combinatorics that go into the definition of the transition functor in the coming sections.
The central notion revolves around \emph{degenerability}, a property that a \emph{formal cube} can have and which is essential to obtaining a simplicial set out of a cubeset.
At first, there are many degenerable formal cubes, but we will see that in order to understand them, it suffices to understand two kinds of cubes we call $C_n$ and $D_n$.

\subsection{Formal Cubes}

Let us describe the formal objects we are interested in for constructing the kernel $K$ of the functor.

Recall that we denote the objects of $\bbox$ by $\{0,1\}^n$ and the objects of $\bbDelta$ by $[n]$ for $n\in\N_0$.

\begin{definition}\label{def:formal-cubes}
A formal $n$-dimensional cube with labels in $[m]$ is a function
\[
    X\colon\{0,1\}^n\to\su[m].
\]
\end{definition}

Note that there is a natural identification of ordered sets
\[
    \su[m] \cong \mathcal{P}(\{0,\dots,m\})\setminus\emptyset
\]
since each subobject of $[m]$ in $\bbDelta$ is uniquely determined by the corresponding subset of $\{0,\ldots,m\}$.

\begin{remark}\label{rem:formal-cubes-label}
To any formal cube $X$ we associate a sequence $f_0,\ldots,f_m$ of functions $f_i\colon\{0,1\}^n\to\{0,1\}$,
the \emph{label functions of $X$}, telling us whether, at a given point $x\in\{0,1\}^n$, the value $X(x)$ contains $i$, i.e.,
\[
    i\in X(x)\quad \iff \quad f_i(x)=1.
\]
The cube $X$ can be uniquely reconstructed from its label functions: simply set
\[
    X(x)=\{i\in\{0,\ldots,m\} \mid f_i(x)=1\}.
\]
So we can think of $X$ in terms of its label functions $f_i$.
We write $x = (x_0,\ldots,x_{n-1}) = (x_j)_{j=0}^{n-1}$ for the points $x\in\{0,1\}^n$, starting with $x_0$,
and call $x_j$ the $j$-th coordinate direction (or variable).
We also write $x_j$ for the projection $\{0,1\}^n\to\{0,1\}$ onto the $j$-th coordinate.
We identify $\{0,1\}$ with the Boolean algebra with two elements and use Boolean operations to denote maps $\{0,1\}^n\to\{0,1\}$.
Also, we use the convention $x_{-1}=1$ and $x_n=0$.
\end{remark}

\begin{example}\label{ex:Cn-Dn}
The cubes $C_n\colon\{0,1\}^n\to\su[n]$ for $n\in\N$ are given by
\[
    f_i = x_i\vee\neg x_{i-1}.
\]
The cubes $D_n\colon\{0,1\}^n\to\su[n]$ for $n\in\N$ are given by
\[
    f_i = x_i\vee\neg x_{i-1}
\]
for $i\leq n-1$ and $f_n = 1$.
\end{example}

There are two natural operations on formal cubes: pullback in direction of $\bbox$ and pushforward in direction of $\bbDelta$.

\begin{definition}\label{def:formal-pull-push}
Given a map $\phi\colon{\bbox}^k\to{\bbox}^n$ in $\bbox$ and a formal cube $X\colon\{0,1\}^n\to\su[m]$,
the \emph{pullback of $X$ along $\phi$} is the formal cube $X\circ\phi\colon\{0,1\}^k\to\su[m]$.

Given a morphism $\psi\colon[m]\to[l]$ in $\bbDelta$ and a formal cube $X\colon\{0,1\}^n\to\su[m]$,
the \emph{pushforward of $X$ along $\psi$} is the formal cube $\psi_\ast\circ X\colon\{0,1\}^n\to\su[l]$,
where $\psi_\ast\colon\su[m]\to\su[l]$ is the induced pushforward of subobjects, the direct image function.
\end{definition}

\begin{example}\label{ex:formal-pull}
In terms of label functions, pullback along a map in $\bbox$ simply is given by precomposing every label function $\{0,1\}^n\to\{0,1\}$ with the corresponding map,
yielding label functions $\{0,1\}^k\to\{0,1\}$.

A particularly important example is given by the outer (lower and upper) faces of codimension 1.
Recall that the subset $[x_j = \epsilon]$ for $\epsilon\in\{0,1\}$ determines a (canonical) subobject of $\{0,1\}^n$ in $\bbox$, previously denoted by $\epsilon_j$ or $\overline{\epsilon}_j$,
depending on $\epsilon$, see \cite{Bihlmaier2026},
which is the $j$-th outer (lower or upper) face of codimension 1.
Then we denote by $f_i(x_j=\eps)$ the pullback of the $f_i$ with the cube face map $\{0,1\}^{n-1}\to\{0,1\}^n$.
Concretely, this means that we plug in $\eps$ for the coordinate $x_j$ and thus get a function that only depends on the other ($n-1$)-many coordinates.

Another example is given by flips: if $\sigma_j\colon\{0,1\}^n\to\{0,1\}^n$ flips the $j$-th coordinate of the cube,
then on label functions this corresponds to inverting every appearance of $x_j$.
\end{example}

\begin{example}\label{ex:formal-push}
We can also describe what label pushforward along a map $\psi\colon[m]\to[l]$ in $\bbDelta$ means for a formal cube $X = (f_i)$.
The label functions $g_i$ of $\psi_\ast\circ X$ are then given by
\[
    g_i = \bigvee_{j\in\psi^{-1}(i)} f_j
\]
where the supremum is taken pointwise and with the convention that the supremum over the empty set is $0$.

For instance, if $\psi = \sigma_i\colon[m]\to[m-1]$ is the $i$-th degeneracy identifying $i$ with $i+1$,
then $\sigma_{i,\ast} X$ has label functions $f_j$ for $0\leq j\leq i-1$, $f_i\vee f_{i+1}$ for $j = i$ and $f_{j+1}$ for $i+1\leq j\leq m-1$.

And if $\psi = \delta_i\colon[m]\to[m+1]$ is the $i$-th face map which is injective and misses the value $i$,
then $\delta_{i,\ast} X$ has label functions $f_j$ for $0\leq j\leq i-1$, $0$ for $j = i$ and $f_{j-1}$ for $i+1\leq j\leq m$.
\end{example}

If two formal $n$-cubes $A$ and $B$ with labels in $[m]$ agree on a face $[x_j = \epsilon]$, we can build a cube $AB$ which shall be thought of as glued along the common face,
compare with \cite[Section 3.1.1]{Bihlmaier2026}.

\begin{definition}\label{def:formal-gluing}
Let $0\le j_0<n$ and $\eps\in\{0,1\}$ such that
\[
    A\rvert_{x_{j_0}=\eps}=B\rvert_{x_{j_0}=\eps}.
\]
Then the gluing $AB$ of $A$ and $B$ is the formal cube
\[
    AB(x) = (x_{j_0}\wedge A\rvert_{x_{j_0}=\neg \eps})\vee(\neg x_{j_0}\wedge B\rvert_{x_{j_0}=\neg \eps})
    = \begin{cases}
        A\rvert_{x_{j_0}=\neg \eps},\quad &\text{if}\; x_{j_0}=1, \\
        B\rvert_{x_{j_0}=\neg \eps},\quad &\text{if}\; x_{j_0}=0.
      \end{cases}
\]
\end{definition}

\begin{remark}\label{rem:formal-gluing-label}
In terms of label functions this definition takes the following form.
Two cubes $A=(f_i)_i$ and $B=(g_i)_i$ can be glued along the face $[x_{j_0}=\eps]$ for $\eps\in\{0,1\}$ if for all $i$,
\[
    f_i(x_{j_0}=\eps)=g_i(x_{j_0}=\eps).
\]
The gluing $AB$ then is given by the label functions
\[
    h_i = (x_{j_0}\wedge f_i(x_{j_0}=\neg\eps))\vee((\neg x_{j_0})\wedge g_i(x_{j_0}=\neg\eps)).
\]
Note that if some coordinate function of $A$ and $B$ agrees, i.e., $f_{i_0} = g_{i_0}$, then the gluing $AB$ has the same label function with $x_{j_0}=\neg\eps$ plugged in.
In particular, if the label function $f_{i_0} = g_{i_0}$ does not depend on $x_{j_0}$, then the label function of the glued cube is the same as the original one.
Thus, when gluing two cubes along a direction $x_{j_0}$, it only matters to look at those label functions that are different or depend on $x_{j_0}$.
In many applications, we will only glue cubes where the label functions are the same at almost every index and only a few depend on $x_{j_0}$, making gluing local.
\end{remark}

\begin{remark}\label{rem:formal-gluing-push}
Gluing formal cubes is stable under pushforward of labels:
if $A$ and $B$ are formal $n$-cubes with labels in $[m]$ that can be glued along the face $[x_{j_0}=\eps]$ and $AB$ denotes their gluing
as in Definition~\ref{def:formal-gluing}, then for any map $\psi\colon[m]\to[k]$ in $\bbDelta$,
the cubes $\psi_\ast A$ and $\psi_\ast B$ can also be glued along the face $[x_{j_0}=\eps]$ and $\psi_\ast(AB) = (\psi_\ast A)(\psi_\ast B)$.
\end{remark}

\begin{remark}
Gluing is in general not associative nor commutative
because we can glue along different faces after one another.
If no brackets are placed, we read composition from right to left.
We don't usually indicate along which faces we have glued
since in practice there is no ambiguity in what compositions mean and along which faces one glues.
\end{remark}

\begin{definition}\label{def:formal-degenerate}
A formal $n$-cube $X$ on label set $[m]$ is \emph{degenerate (along the $j$-th coordinate direction)}
if there exists a formal ($n-1$)-cube $Y$ on label set $[m]$ such that $X = p_j^\ast Y = Y\circ p_j$
where $p_j\colon\{0,1\}^n\to\{0,1\}^{n-1}$ is the $j$-th coordinate projection.
\end{definition}

\begin{remark}\label{rem:formal-degenerate-label}
Equivalently, a formal cube is degenerate if it doesn't depend on the $j$-th coordinate.
For its label functions this means that every label function doesn't depend on the $j$-th coordinate.
\end{remark}

\subsection{Classification of Degenerable Cubes}\label{ssec:classification}

The central class of $n$-dimensional cubes on the label set $[n]$ is the following.
This definition is made so that in the resulting simplicial set of a cubeset there exist degeneracies,
and not only face maps.

\begin{definition}\label{def:formal-degenerable}
A formal $n$-cube on label set $[n]$ is \emph{degenerable}
if for every $0\le i<n$ pushforward along $\sigma_i$ yields a cube degenerate along the $i$-th coordinate direction.
\end{definition}

\begin{remark}\label{rem:formal-degenerable-label}
In formulas, a cube is degenerable if and only if for every $0\le i<n$, the function
\[
    f_i\vee f_{i+1}\colon\{0,1\}^n\to\{0,1\}
\]
does not depend on $x_i$.
\end{remark}

\begin{remark}
Of the automorphism group of the cubes, only flips preserve degenerability:
if $X$ is a degenerable cube and $\sigma\colon\{0,1\}^n\to\{0,1\}^n$ a (composition of) flip,
then $X\circ\sigma$ is degenerable as well.
But coordinate permutations don't preserve this property,
so one could be tempted to define a notion being closed under automorphisms.
However, we will avoid these anyway and believe that the theory is cleaner in this way.

Also, the label pushforward of a degenerable cube $X$ mostly results in a non-degenerable cube,
which often is degenerate (take, e.g., $\sigma_{i,\ast} X$ for a degenerable cube $X$).
\end{remark}

\begin{remark}
Gluing something degenerate to something non-degenerate can yield a degenerate formal cube,
and gluing two degenerate formal cubes can yield a non-degenerate formal cube.
\end{remark}

Now we classify how the label functions of a degenerable cube can look like.
Recall the convention that $x_{-1}=1$ and $x_{n}=0$.
The key insight is the following observation.

\begin{lemma}
The $i$-th label function $f_i$ of a degenerable formal $n$-cube does only depend on the coordinates $x_i$ and $x_{i-1}$.
\end{lemma}

\begin{proof}
For all $j\notin \{i-1,i\}$, collapsing the labels $j$ and $j+1$ does not affect the value of $f_i$.
Since it must result in a cube which is degenerate in the $j$-th coordinate direction,
$f_i$ needs to be independent of $x_j$.
\end{proof}

This means that the label function factors over the map of cubes $\{0,1\}^n\to\{0,1\}^2$ and thus we also write $f_i=f_i(x_{i-1},x_i)$, discarding all other coordinate directions.

\begin{corollary}
The label functions $f_i$ of a degenerable formal cube only have 16 possibilities, corresponding to all maps $\{0,1\}^2\to\{0,1\}$.
They can be expressed by the following Boolean operations in two variables:
\begin{align*}
    \{&0,1,x_i,\neg x_i, x_{i-1}, \neg x_{i-1}, x_i \wedge x_{i-1}, \neg x_i \wedge x_{i-1}, x_i\wedge \neg x_{i-1}, \neg x_i \wedge \neg x_{i-1}, x_i\vee x_{i-1},	\\
    & \neg x_i \vee x_{i-1}, x_{i}\vee \neg x_{i-1}, \neg x_i\vee \neg x_{i-1}, (x_i\wedge x_{i-1})\vee (\neg x_i\wedge \neg x_{i-1}), (x_i\wedge \neg x_{i-1})\vee (\neg x_i\wedge x_{i-1})\}
\end{align*}
\end{corollary}

We also denote $(x_i\wedge x_{i-1})\vee (\neg x_i\wedge \neg x_{i-1})$ by $x_i\iff x_{i-1}$
and its negation $(x_i\wedge \neg x_{i-1})\vee (\neg x_i\wedge x_{i-1})$ by $\neg(x_i\iff x_{i-1})$.
	
\begin{remark}
The starting coordinate convention $x_{-1}=1$ actually narrows $f_0$ to be of the form
\[
    f_0\in \{0,1,x_{0},\neg x_0\}
\]
and the endpoint convention $x_n=0$ forces
\[
    f_{n}\in \{0,1,x_{n-1},\neg x_{n-1}\}.
\]
\end{remark}

\begin{corollary}
A formal cube is degenerable if and only if for all $0\leq i<n$ its label functions $f_i$ only depend on the coordinates $x_{i-1},x_i$ and
\[
    f_i\vee f_{i+1}\colon \{0,1\}^3\to \{0,1\}
\]
is independent of $x_i$.
\end{corollary}

\begin{remark}
This reduction makes gluing along coordinate $x_i$ obviously purely local in labels $f_i$ and $f_{i+1}$,
see Remark~\ref{rem:formal-gluing-label}.
Thus, gluing is often very explicit and easy in those particular cases.
\end{remark}

The above constraint reduces the allowed combinations and makes it possible to classify all combinations.

\begin{lemma}
A formal cube $(f_i)$ is degenerable if and only if for every $0\le i<n$ the relation of $f_i$ and $f_{i+1}$ is one of the following:
\begin{enumerate}[(i)]
    \item If $f_i=1$, then $f_{i+1}$ is an arbitrary function in the variables $x_{i+1}$ and $x_i$,
	\item If $f_i\in \{0,x_{i-1},\neg x_{i-1}\}$, then $f_{i+1}\in \{0,1,x_{i+1},\neg x_{i+1}\}$,
    \item If $f_i\in \{x_i, x_i\vee x_{i-1}, x_i \vee \neg x_{i-1}\}$, then $f_{i+1}\in \{1, \neg x_i, \neg x_i \vee x_{i+1}, \neg x_i \vee \neg x_{i+1}\}$,
	\item If $f_i \in \{\neg x_i, \neg x_i\vee x_{i-1}, \neg x_{i}\vee \neg x_{i-1} \}$, then $f_{i+1}\in \{1,x_i, x_i\vee x_{i+1}, x_i \vee \neg x_{i+1}\}$,
    \item If $f_i\in \{x_i\wedge x_{i-1},x_i\wedge\neg x_{i-1}, \neg x_i\wedge x_{i-1}, \neg x_i\wedge \neg x_{i-1}, x_i\iff x_{i-1}, \neg(x_i\iff x_{i-1})\}$, then $f_{i+1}=1$.
\end{enumerate}
\end{lemma}

\begin{proof}
It is straightforward to check that each of these combinations of $f_i$ and $f_{i+1}$ yields a degenerable cube,
i.e., $f_i\vee f_{i+1}$ is independent of $x_i$.
So it remains to show that these are precisely the cases where $f_i\vee f_{i+1}$ is independent of $x_i$.
\begin{enumerate}[(i)]
	\item The first case is clear.
	\item This means that $f_i$ is independent of $x_i$ and not constant $1$, so let us assume without loss of generality that $f_i(0,x_i)=0$.
		Now the $x_i$ independence of $(f_i\vee f_{i+1})(0,x_i,x_{i+1})=f_{i+1}(x_{i},x_{i+1})$ directly implies $x_i$-independence of $f_{i+1}$.
	\item In this case for every choice of $x_{i-1}$ there exists an $x_i$ such that $f_{i}(x_{i-1},x_i)=1$.
        This implies that the same is true for $f_{i}\vee f_{i+1}$. But by independence of $x_i$ this implies that $f_i \vee f_{i+1}=1$.
		Moreover, in each of the formulas for $f_i$, there is $\eps\in\{0,1\}$ such that $f_{i}(\eps,0)=0$.
		Hence $f_{i+1}$ has to be $1$ whenever $x_i=0$, resulting in this family of formulas.
    \item This follows from the third case by switching $x_i$ and $\neg x_i$.
	\item The last case achieves for both values of $x_{i}$ both values $1$ and $0$, depending on $x_{i-1}$.
	Now assume $f_{i+1}$ were not the constant $1$-function.
	Without loss of generality (flip $x_i$ or $x_{i+1}$ to $\neg x_i$ resp. $\neg x_{i+1}$ if necessary),
	let $f_{i+1}(0,0)=0$.
	Now choose $x_{i-1}$ in such a way that $f_i(x_{i-1},0)=0$ and accordingly $f_i(x_{i-1},1)=1$ (again without loss of generality this happens for $x_{i-1}=0$).
 	Then
	\[(f_i\vee f_{i+1})(0,0,0)=0 \quad\text{and}\quad (f_i\vee f_{i+1})(0,1,0)=1\]
	contradicting the independence of $x_i$. Thus the last case is clear. \qedhere
\end{enumerate}
\end{proof}

Therefore, given a degenerable cube, we see that one can always assume (after inductively applying flips to the cube, which act as negation on the corresponding $x_j$) that in $f_i$,
the $i$-th coordinate appears as $x_i$ and the ($i-1$)-st coordinate appears as $\neg x_{i-1}$,
apart from the case where $f_i = (x_i\iff x_{i-1})$ or $f_i = \neg (x_i\iff x_{i-1})$.
But here we just make the choice of $x_i\iff x_{i-1}$ being the canonical one.
This choice implies that we have a very strong reduced normal form of degenerable cubes because we may eliminate half of the expressions.
Thus we have proven the following proposition.

\begin{proposition}\label{prop:class-formal-degenerable}
Up to flips, a degenerable formal cube consists of a unique%
\footnote{up to the endpoints $f_0$ and $f_n$, where the convention $x_{-1} = 1$ and $x_n = 0$ identifies some label functions}
sequence $(f_i)_{i=0}^n$ of label functions
\[
    f_i\in\{0,1,x_i, \neg x_{i-1}, x_i \wedge \neg x_{i-1}, x_i\vee \neg x_{i-1}, x_i\iff x_{i-1}\}
\]
that fulfill the following recursion.
\begin{enumerate}[(i)]
    \item If $f_i=1$, then $f_{i+1}$ can be chosen arbitrarily,
	\item If $f_i\in \{0,\neg x_{i-1}\}$, then $f_{i+1}\in \{0,1,x_{i+1}\}$,
    \item If $f_i\in \{x_i, x_i \vee \neg x_{i-1}\}$, then $f_{i+1}\in \{1, \neg x_i, \neg x_i \vee x_{i+1}\}$,
	\item If $f_i\in \{x_i\wedge \neg x_{i-1}, x_i\iff x_{i-1}\}$, then $f_{i+1}=1$.
\end{enumerate}
\end{proposition}

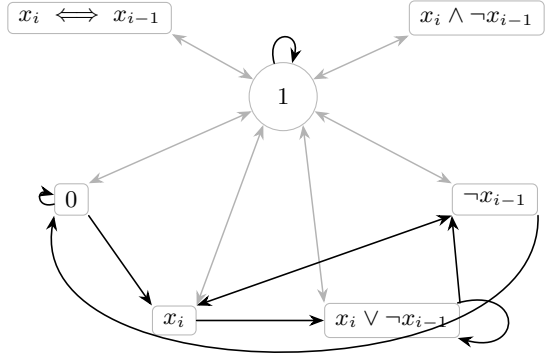
\begin{center}
\begin{tikzpicture}[
    font=\small,
    line cap=round,
    line join=round,
    state/.style={
        draw=gray!55,
        rounded corners=2pt,
        fill=white,
        inner xsep=4pt,
        inner ysep=3pt
    },
    one/.style={
        state,
        circle,
        minimum size=9mm,
        inner sep=1pt
    },
    transition/.style={
        ->,
        >=Stealth,
        line width=0.6pt
    },
    bidirectional/.style={
        <->,
        >=Stealth,
        line width=0.6pt
    },
    universal/.style={
        bidirectional,
        draw=gray!60,
        line width=0.55pt
    },
scale=0.8
]

\node[state] (iff)  at (-3.2, 3.0)
    {$x_i\iff x_{i-1}$};
\node[state] (and)  at ( 3.2, 3.0)
    {$x_i\wedge\neg x_{i-1}$};
\node[one]   (one)  at ( 0.0, 1.7) {$1$};

\node[state] (zero) at (-3.5, 0.0) {$0$};
\node[state] (neg)  at ( 3.5, 0.0)
    {$\neg x_{i-1}$};
\node[state] (x)    at (-1.8,-2.0) {$x_i$};
\node[state] (vee)  at ( 1.8,-2.0)
    {$x_i\vee\neg x_{i-1}$};

\draw[universal] (one.150) -- (iff.south east);
\draw[universal] (one.30)  -- (and.south west);
\draw[universal] (one.205) -- (zero.north east);
\draw[universal] (one.335) -- (neg.north west);
\draw[universal] (one.235) -- (x.north east);
\draw[universal] (one.305) -- (vee.north west);

\path[transition] (one)
    edge[loop above,looseness=5] (one);
\path[transition] (zero)
    edge[loop left,looseness=5] (zero);
\path[transition] (vee)
    edge[loop right,looseness=5] (vee);

\draw[transition]
    (zero.south east) -- (x.north west);
\draw[transition]
    (x.east) -- (vee.west);
\draw[transition]
    (vee.north east) -- (neg.south west);
\draw[bidirectional]
    (x.north east) -- (neg.south west);

\draw[transition]
    (neg.south east)
    .. controls (4.3,-3.25) and (-4.3,-3.25) ..
    (zero.south west);
\end{tikzpicture}
\end{center}

Now let us remark that many of the above combinations directly imply degeneracy of the cube, since then some $x_j$ does not appear.
For example, if $f_0 = 0$, then by the recursion condition, $f_1$ doesn't depend on $x_0$ as well,
and so the cube is already degenerate, since no other $f_i$ for $i\geq 2$ can depend on $x_0$.
So the non-degenerate degenerable formal cubes have an even simpler description.

\begin{proposition}\label{prop:class-formal-degenerable-non-degenerate}
The non-degenerate degenerable formal cubes are, up to flips, precisely the ones where the recursion is sharpened to
\begin{enumerate}[(i)]
    \item We have $f_0\in\{1,x_0\vee\neg x_{-1}\}$,
    \item If $f_i=1$, then $f_{i+1}\in \{x_{i+1} \wedge \neg x_{i}, x_{i+1}\vee \neg x_{i}, x_{i+1}\iff x_{i}\}$,
	\item If $f_i\in \{x_i \vee \neg x_{i-1}\}$, then $f_{i+1}\in \{1, x_{i+1}\vee\neg x_i\}$,
	\item If $f_i\in \{x_i\wedge \neg x_{i-1}, x_i\iff x_{i-1}\}$, then $f_{i+1}=1$,
	\item We have $f_n\in\{1,x_n\vee\neg x_{n-1}\}$.
\end{enumerate}
In particular, the cases $f_i\in\{0,x_i,\neg x_{i-1}\}$ do not occur
(except on the endpoints where $x_0 = x_0\vee\neg x_{-1}$ and $\neg x_{n-1} = x_n\vee\neg x_{n-1}$).
\end{proposition}
\begin{center}
\begin{tikzpicture}[
    font=\small,
    line cap=round,
    line join=round,
    state/.style={
        draw=gray!55,
        rounded corners=2pt,
        fill=white,
        inner xsep=4pt,
        inner ysep=3pt
    },
    endpoint/.style={
        state,
        fill=gray!12
    },
    one/.style={
        endpoint,
        circle,
        minimum size=9mm,
        inner sep=1pt
    },
    transition/.style={
        <->,
        >=Stealth,
        line width=0.6pt
    },
    loop/.style={
        ->,
        >=Stealth,
        line width=0.6pt
    },
scale=0.8
]

\node[state] (iff) at (-3,1.8)
    {$x_i\iff x_{i-1}$};
\node[state] (and) at (3,1.8)
    {$x_i\wedge\neg x_{i-1}$};
\node[one] (one) at (0,0.8) {$1$};
\node[endpoint] (vee) at (0,-1.5)
    {$x_i\vee\neg x_{i-1}$};

\draw[transition]
    (one.155) -- (iff.south east);
\draw[transition]
    (one.25) -- (and.south west);
\draw[transition]
    (one.south) -- (vee.north);

\path[loop]
    (vee) edge[loop right,looseness=5] (vee);
\end{tikzpicture}
\end{center}

\begin{example}
The cubes $C_n$ and $D_n$ from Example~\ref{ex:Cn-Dn} are non-degenerate degenerable formal cubes
fulfilling the normal form.
\end{example}

The normal form also implies the following description of non-degenerate degenerable formal cubes.

\begin{corollary}\label{cor:class-formal-degenerable-non-degenerate-interval}
A non-degenerate degenerable formal $n$-cube consists of a unique partition of $\{0,\dots, n\}$ into nonempty intervals $I_1,\ldots, I_k$
such that, up to flips, for each interval $I=\{a,\dots, b\}$ we have that
\begin{itemize}
\item if $a = 0$, then either
\begin{itemize}
    \item $f_i = x_i\vee\neg x_{i-1}$ for all $0\leq i\leq b$
    \item or $f_0 = 1$ and we have that
    \begin{itemize}
    	\item either $b=1$ and $f_1\in\{x_1\wedge\neg x_0, x_0\iff x_1\}$
        \item or $b\geq 1$ and $f_i=x_i\vee \neg x_{i-1}$ for all $0<i\le b$.
    \end{itemize}
\end{itemize}
\item and if $a > 0$, then we have that $f_a = 1$ and
\begin{itemize}
	\item either $b=a+1$ and $f_b\in\{x_b\wedge\neg x_{a}, x_a\iff x_b\}$
    \item or $b\geq a+1$ and $f_i=x_i\vee \neg x_{i-1}$ for all $a<i\le b$
    \item or $a = b = n$ and $f_n = 1$.
\end{itemize}
\end{itemize}
\end{corollary}

\begin{corollary}
If a degenerable $n$-cube has only up to $n-2$-tuples as labels, then it is degenerate.
\end{corollary}

\begin{proof}
This follows by contraposition from Corollary~\ref{cor:class-formal-degenerable-non-degenerate-interval}:
if $(f_i)$ constitute a non-degenerate degenerable formal cube,
one can inductively construct an $n$-tuple $x\in\{0,1\}^n$ such that $f_i(x) = 1$ for at least ($n-1$)-many $i$.
\end{proof}

\subsection{Generation of Degenerable Cubes}

The previous description allows us to see the following uniqueness result.

\begin{proposition}\label{prop:unique-Cn}
For every $n\in\N$ there exists, up to flips, exactly one degenerable formal $n$-cube without a full label that has every $n$-tuple precisely once as a label.
\end{proposition}

\begin{proof}
Existence: take the cube $C_n$ from Example~\ref{ex:Cn-Dn}.
For each $A\subset\{0,\ldots,n\}$ with $\lvert A\rvert = n$ we have that the equations $f_i(x) = 1$ if and only if $i\in A$
have exactly one solution: start from $f_0 = x_0$ and inductively construct the other $x_j$.
Thus, each $n$-tuple appears exactly once in $C_n$ as a label.
But the full label doesn't appear in $C_n$ since $f_i(x) = 1$ for all $0\leq i\leq n$ doesn't have a solution:
necessarily, $x_0 = x_1 = \ldots = x_{n-1} = 1$, but then $f_n(x) = \neg x_{n-1} = 0$.

Uniqueness: assume we have constructed such a cube $(f_i)$.
In a degenerate cube, every label appears at least twice, thus $(f_i)$ isn't degenerate.
Therefore, we can apply Corollary~\ref{cor:class-formal-degenerable-non-degenerate-interval} and assume that it is of the form given there.
If the cube would contain a label function that is constant $1$, it necessarily contains the full label:
then it is possible to find a solution to the equations $f_i(x) = 1$ for $i\in I$ in every interval $I$ of the partition of $\{0,\ldots,n\}$
(going through all cases in Corollary~\ref{cor:class-formal-degenerable-non-degenerate-interval}).
Therefore, $I = \{0,\ldots,n\}$ and $f_i = x_i\vee\neg x_{i-1}$.
\end{proof}

Moreover, in the previous proposition, we can also swap the condition on the $n$-tuples for non-degeneracy of the cubes and obtain essentially the same result.

\begin{proposition}
If a degenerable formal cube is nondegenerate and has no full label then it is a flip of $C_n$.
\end{proposition}

\begin{proof}
One applies Corollary~\ref{cor:class-formal-degenerable-non-degenerate-interval} and argues as in the proof of Proposition~\ref{prop:unique-Cn}
to see that $f_i = x_i\vee\neg x_{i-1}$.
\end{proof}

\begin{definition}
Define $C_n$ for $n\in\N$ to be the unique formal $n$-cube from Proposition~\ref{prop:unique-Cn},
oriented such that the value at $0$ is $\hat{0}=\{1,\ldots,n\}$, see also Example~\ref{ex:Cn-Dn}.
Define $D_n$ to be the cube $C_n$ with the label $n$ attached at every point,
i.e., $f_i^D=f_i^C$ for $i<n$ and $f_n^D=1$, see Example~\ref{ex:Cn-Dn}.
We define $D_0$ by $f_0=1$.
\end{definition}

\begin{example}
The low-dimensional $C_n$ look as follows.
\begin{center}
\hspace{-2.1cm}
\begin{tikzpicture}[
    scale=0.92,
    transform shape,
    edge/.style={draw=black,line width=0.55pt},
    hiddenedge/.style={draw=gray!55,line width=0.45pt},
    connector/.style={draw=gray!65,line width=0.40pt},
    vlabel/.style={font=\scriptsize,fill=white,inner sep=1pt},
    head/.style={font=\small\bfseries}
]

\useasboundingbox (-0.4,-0.75) rectangle (14.6,4.25);

\begin{scope}[shift={(0.35,2.10)}]
  \coordinate (A) at (0,0);
  \coordinate (B) at (1.8,0);
  \draw[edge] (A)--(B);
  \node[vlabel] at (A) {$0$};
  \node[vlabel] at (B) {$1$};
\end{scope}

\begin{scope}[shift={(3.10,1.35)}]
  \coordinate (A) at (0,0);
  \coordinate (B) at (2,0);
  \coordinate (C) at (0,2);
  \coordinate (D) at (2,2);

  \draw[edge] (A)--(B)--(D)--(C)--cycle;

  \node[vlabel] at (A) {$01$};
  \node[vlabel] at (B) {$02$};
  \node[vlabel] at (C) {$1$};
  \node[vlabel] at (D) {$12$};
\end{scope}

\begin{scope}[shift={(6.35,1.10)}]
  \coordinate (A) at (0,0);
  \coordinate (B) at (2.1,0);
  \coordinate (C) at (0,2.1);
  \coordinate (D) at (2.1,2.1);
  \coordinate (E) at (0.75,0.75);
  \coordinate (F) at (2.85,0.75);
  \coordinate (G) at (0.75,2.85);
  \coordinate (H) at (2.85,2.85);

  \draw[edge] (A)--(B)--(D)--(C)--cycle;
  \draw[edge] (E)--(F)--(H)--(G)--cycle;
  \draw[hiddenedge] (A)--(E);
  \draw[edge] (B)--(F);
  \draw[edge] (C)--(G);
  \draw[edge] (D)--(H);

  \node[vlabel] at (C) {$13$};
  \node[vlabel] at (D) {$123$};
  \node[vlabel] at (G) {$013$};
  \node[vlabel] at (H) {$023$};

  \node[vlabel] at (A) {$12$};
  \node[vlabel] at (B) {$12$};
  \node[vlabel] at (E) {$012$};
  \node[vlabel] at (F) {$02$};
\end{scope}

\begin{scope}[
    shift={(10,0.60)},
    scale=0.66,
    transform shape,
    vlabel/.style={font=\footnotesize,fill=white,inner sep=1pt}
]
  \coordinate (A) at (0,0);
  \coordinate (B) at (4.5,0);
  \coordinate (C) at (0,3.8);
  \coordinate (D) at (4.5,3.8);
  \coordinate (E) at (1.25,1.25);
  \coordinate (F) at (5.75,1.25);
  \coordinate (G) at (1.25,5.05);
  \coordinate (H) at (5.75,5.05);

  \coordinate (a) at (1.95,1.45);
  \coordinate (b) at (3.85,1.45);
  \coordinate (c) at (1.95,3.05);
  \coordinate (d) at (3.85,3.05);
  \coordinate (e) at (2.40,1.90);
  \coordinate (f) at (4.30,1.90);
  \coordinate (g) at (2.40,3.50);
  \coordinate (h) at (4.30,3.50);

  \draw[edge]       (A)--(B)--(D)--(C)--cycle;
  \draw[hiddenedge] (A)--(E);
  \draw[hiddenedge] (E)--(F);
  \draw[edge]       (F)--(H)--(G);
  \draw[hiddenedge] (G)--(E);
  \draw[edge]       (B)--(F);
  \draw[edge]       (C)--(G);
  \draw[edge]       (D)--(H);

  \draw[edge] (a)--(b)--(d)--(c)--cycle;
  \draw[edge] (e)--(f)--(h)--(g)--cycle;
  \draw[edge] (a)--(e);
  \draw[edge] (b)--(f);
  \draw[edge] (c)--(g);
  \draw[edge] (d)--(h);

  \draw[connector] (A)--(a);
  \draw[connector] (B)--(b);
  \draw[connector] (C)--(c);
  \draw[connector] (D)--(d);
  \draw[connector] (E)--(e);
  \draw[connector] (F)--(f);
  \draw[connector] (G)--(g);
  \draw[connector] (H)--(h);

  \node[vlabel] at (A) {$1234$};
  \node[vlabel] at (B) {$0234$};
  \node[vlabel] at (C) {$134$};
  \node[vlabel] at (D) {$0134$};

  \node[vlabel] at (E) {$124$};
  \node[vlabel] at (F) {$024$};
  \node[vlabel] at (G) {$124$};
  \node[vlabel] at (H) {$0124$};

  \node[vlabel] at (a) {$123$};
  \node[vlabel] at (b) {$023$};
  \node[vlabel] at (c) {$13$};
  \node[vlabel] at (d) {$013$};

  \node[vlabel] at (e) {$123$};
  \node[vlabel] at (f) {$023$};
  \node[vlabel] at (g) {$123$};
  \node[vlabel] at (h) {$0123$};
\end{scope}

\node[head] at (7.2,0) {The cubes \(C_n\)};

\end{tikzpicture}
\end{center}
And the first $D_n$ are as follows (note that $D_0$ exists, in contrast to $C_0$).

\begin{center}
\begin{tikzpicture}[
    scale=0.92,
    transform shape,
    edge/.style={draw=black,line width=0.55pt},
    hiddenedge/.style={draw=gray!55,line width=0.45pt},
    connector/.style={draw=gray!65,line width=0.40pt},
    vlabel/.style={font=\scriptsize,fill=white,inner sep=1pt},
    head/.style={font=\small\bfseries}
]

\useasboundingbox (-0.4,-0.75) rectangle (14.6,4.25);
\begin{scope}[shift={(-1.2,2.1)}]
	\coordinate (A) at (0,0);
	\node[vlabel] at (A) {$0$};
\end{scope}

\begin{scope}[shift={(0.35,2.10)}]
  \coordinate (A) at (0,0);
  \coordinate (B) at (1.8,0);
  \draw[edge] (A)--(B);
  \node[vlabel] at (A) {$01$};
  \node[vlabel] at (B) {$1$};
\end{scope}

\begin{scope}[shift={(3.10,1.35)}]
  \coordinate (A) at (0,0);
  \coordinate (B) at (2,0);
  \coordinate (C) at (0,2);
  \coordinate (D) at (2,2);

  \draw[edge] (A)--(B)--(D)--(C)--cycle;

  \node[vlabel] at (A) {$012$};
  \node[vlabel] at (B) {$02$};
  \node[vlabel] at (C) {$12$};
  \node[vlabel] at (D) {$12$};
\end{scope}

\begin{scope}[shift={(6.35,1.10)}]
  \coordinate (A) at (0,0);
  \coordinate (B) at (2.1,0);
  \coordinate (C) at (0,2.1);
  \coordinate (D) at (2.1,2.1);
  \coordinate (E) at (0.75,0.75);
  \coordinate (F) at (2.85,0.75);
  \coordinate (G) at (0.75,2.85);
  \coordinate (H) at (2.85,2.85);

  \draw[edge] (A)--(B)--(D)--(C)--cycle;
  \draw[edge] (E)--(F)--(H)--(G)--cycle;
  \draw[hiddenedge] (A)--(E);
  \draw[edge] (B)--(F);
  \draw[edge] (C)--(G);
  \draw[edge] (D)--(H);

  \node[vlabel] at (C) {$13$};
  \node[vlabel] at (D) {$123$};
  \node[vlabel] at (G) {$013$};
  \node[vlabel] at (H) {$023$};

  \node[vlabel] at (A) {$123$};
  \node[vlabel] at (B) {$123$};
  \node[vlabel] at (E) {$0123$};
  \node[vlabel] at (F) {$023$};
\end{scope}

\begin{scope}[
    shift={(10,0.60)},
    scale=0.66,
    transform shape,
    vlabel/.style={font=\footnotesize,fill=white,inner sep=1pt}
]
  \coordinate (A) at (0,0);
  \coordinate (B) at (4.5,0);
  \coordinate (C) at (0,3.8);
  \coordinate (D) at (4.5,3.8);
  \coordinate (E) at (1.25,1.25);
  \coordinate (F) at (5.75,1.25);
  \coordinate (G) at (1.25,5.05);
  \coordinate (H) at (5.75,5.05);

  \coordinate (a) at (1.95,1.45);
  \coordinate (b) at (3.85,1.45);
  \coordinate (c) at (1.95,3.05);
  \coordinate (d) at (3.85,3.05);
  \coordinate (e) at (2.40,1.90);
  \coordinate (f) at (4.30,1.90);
  \coordinate (g) at (2.40,3.50);
  \coordinate (h) at (4.30,3.50);

  \draw[edge]       (A)--(B)--(D)--(C)--cycle;
  \draw[hiddenedge] (A)--(E);
  \draw[hiddenedge] (E)--(F);
  \draw[edge]       (F)--(H)--(G);
  \draw[hiddenedge] (G)--(E);
  \draw[edge]       (B)--(F);
  \draw[edge]       (C)--(G);
  \draw[edge]       (D)--(H);

  \draw[edge] (a)--(b)--(d)--(c)--cycle;
  \draw[edge] (e)--(f)--(h)--(g)--cycle;
  \draw[edge] (a)--(e);
  \draw[edge] (b)--(f);
  \draw[edge] (c)--(g);
  \draw[edge] (d)--(h);

  \draw[connector] (A)--(a);
  \draw[connector] (B)--(b);
  \draw[connector] (C)--(c);
  \draw[connector] (D)--(d);
  \draw[connector] (E)--(e);
  \draw[connector] (F)--(f);
  \draw[connector] (G)--(g);
  \draw[connector] (H)--(h);

  \node[vlabel] at (A) {$1234$};
  \node[vlabel] at (B) {$0234$};
  \node[vlabel] at (C) {$134$};
  \node[vlabel] at (D) {$0134$};

  \node[vlabel] at (E) {$124$};
  \node[vlabel] at (F) {$024$};
  \node[vlabel] at (G) {$124$};
  \node[vlabel] at (H) {$0124$};

  \node[vlabel] at (a) {$1234$};
  \node[vlabel] at (b) {$0234$};
  \node[vlabel] at (c) {$134$};
  \node[vlabel] at (d) {$0134$};

  \node[vlabel] at (e) {$1234$};
  \node[vlabel] at (f) {$0234$};
  \node[vlabel] at (g) {$1234$};
  \node[vlabel] at (h) {$01234$};
\end{scope}

\node[head] at (7.2,-0) {The cubes \(D_n\)};

\end{tikzpicture}
\end{center}
\end{example}

\begin{remark}
Note that one has the following recursive construction of the $C_n$ and $D_n$.
\begin{enumerate}[(i)]
\item 
The $x_n=0$ and $x_n=1$ faces of $C_{n+1}$ are $C_n$ with the label $n+1$ appended everywhere and $D_n$ as opposite face.
\item 
The $x_n=0$ and $x_n=1$ faces of $D_{n+1}$ are $C_n$ with the label $n+1$ appended everywhere and $D_n$ with the label $n+1$ appended everywhere as opposite face.

\end{enumerate}
Similar recursive decompositions are possible for all degenerable cubes.
\end{remark}

The cubes $C_n$ and $D_n$ are central for the non-degenerate degenerable cubes due to the following reason.

\begin{theorem}\label{prop:C-D-glue-all}
All non-degenerate degenerable formal $n$-cubes are generated, under flips and gluings, by $C_n$ and $D_n$.
\end{theorem}

To prove this theorem we need two lemmas.

\begin{lemma}\label{lem:insert_1s}
For any subset $I\sub \{0,\dots ,n\}$ denote by $D_I$ the degenerable cube with label functions
\[
    f_i = x_i\vee \neg x_{i-1} \vee (i\in I) =
    \begin{cases}
        1,\quad & i\in I, \\
        x_i\vee\neg x_{i-1}\quad & i\notin I.
    \end{cases}
\]
i.e., it is the formula of $C_n$ but with a $1$ inserted at the coordinates in $I$.
Then $D_I$ is generated by $C_n=D_\emptyset$ and $D_n=D_{\{n\}}$ under gluing.
\end{lemma}

\begin{proof}
	We prove this by induction on the size of $I\setminus \{n\}$.
	If $I\setminus \{n\}$ is empty, nothing is to show, since $I=\emptyset$ corresponds to $C_n$ and $I=\{n\}$ corresponds to $D_n$.
	For the induction step, let $k$ be the largest element of $I\setminus \{n\}$ and set $J=I\setminus \{k\}$.
	By induction hypothesis, both $D_J$ and $D_{J\cup\{n\}}$ are generated by $C_n$ and $D_n$ under gluing.
	Now start with $D_{J\cup \{n\}}$ and glue to it $D_J$ ($n-k$)-times, starting with $x_{n-1}$.
	In the first step, let $W^0=D_{J\cup \{n\}}$.
	If $n\in I$, then $n\in J$ and the last two label functions look like
	\begin{align*}
		W^0=(&x_{n-1}\vee \neg x_{n-2}, &1) \\
		D_J=(&x_{n-1}\vee \neg x_{n-2}, &1)
	\end{align*}
	(since $n-1\notin J$) and if $n\notin I$, then
	\begin{align*}
	D_J=(&x_{n-1}\vee \neg x_{n-2}, & \neg x_{n-1}).
	\end{align*}
	Moreover, the label functions agree for $i\leq n-2$.
    In both cases, $W^0$ and $D_J$ agree on the face $[x_{n-1}=0]$.
    So we can glue them along $x_{n-1}$, resulting in $W^1$ with label functions
	\[f_{n-1}=(x_{n-1}\wedge 1) \vee (\neg x_{n-1}\wedge 1)=1 \]
	and, if $n\in I$, then $f_n=1$, and otherwise, for $n\notin I$, we have that
	\[f_n=(x_{n-1}\wedge \neg x_{n-1})\vee (\neg x_{n-1}\wedge 1)=\neg x_{n-1}.\]
	So in both cases the formula for $D_J$ is obtained in the $n$-th label.
	Now we can iterate, and in each step $i\le n-k$ define $W^{i}$ as the gluing of $D_J$ and $W^{i-1}$ in direction $x_{n-i}$,
	where the ($n-i$)-th and ($n-i+1$)-th label functions are given by
	\begin{align*}
		W^{i-1}=(&x_{n-i}\vee \neg x_{n-i-1}, & 1), \\
		D_J=(&x_{n-i}\vee \neg x_{n-i-1}, & x_{n-i+1}\vee \neg x_{n-i})
	\end{align*}
	since $n-i\geq k$ and hence $n-i\notin J$,
	and are equal otherwise.
	Gluing these along $[x_{n-i}=0]$ yields label functions
	\[f_{n-i+1}=(x_{n-i}\wedge x_{n-i+1})\vee \neg x_{n-i}= x_{n-i+1}\vee \neg x_{n-i}\]
	and
	\[f_{n-i}=(x_{n-i}\wedge 1) \vee (\neg x_{n-i}\wedge 1)=1\]
	as formulas for $W^i$, and thus $W^{n-k}=D_I$ is constructed.
\end{proof}

\begin{remark}\label{rem:insert_1s}
Note that if $I = \{k\}$ is a singleton then the gluing procedure to obtain $D_{\{k\}}$ takes a particularly easy form.
In this case, $J = \emptyset$ and therefore, $D_J = C_n$ and $D_{J\cup\{n\}} = D_n$.
Thus, $D_{\{k\}}$ arises as the ($n-k$)-fold gluing of $C_n$ to $D_n$ along the axis $x_{n-i}$ for $1\leq i\leq n-k$,
i.e., $D_n$ appears once and $C_n$ appears ($n-k$)-times.
\end{remark}

\begin{remark}\label{rem:insert-1s-pushforward}
Pushforwarding $C_n$ along a surjection $\phi\colon[n]\twoheadrightarrow [k]$ yields the
cube $\phi_\ast C_n$ with label functions
\[
    f_i=\sup_{j\in \phi^{-1}(i)}x_j\vee \neg x_{j-1}=
    \begin{cases}
        x_{\phi^{-1}(i)}\vee \neg x_{\phi^{-1}(i)-1},\quad & \text{if}\;\lvert\phi^{-1}(i)\rvert =1, \\
        1,\quad&\text{else}.
    \end{cases}
\]
Let $I = \{i\in[k] : \lvert\phi^{-1}(i)\rvert>1\}\subset\{0,\ldots,k\}$
and consider the map of cubes $\psi\colon\{0,1\}^n\to\{0,1\}^k$ given by
\[
    \psi_i(x) = x_{\max\phi^{-1}(i)}.
\]
Then we have that $\phi_\ast C_n = D_I\circ\psi$.
Therefore, $\phi_\ast C_n$ is the $k$-cube $D_I$ blown up in directions of multiple fibers
to a degenerate $n$-cube.

On the other hand, every subset $I\subset[k]$ can be realized as $I = \{i\in [k]:\, |\phi^{-1}(i)|>1\}$
for some surjection $\phi\colon[n]\twoheadrightarrow[k]$ and so $D_I$ is $\phi_\ast C_n$ after potentially enlarging the cube size:
take a right inverse $\eta$ of the surjection $\psi$ and obtain $\phi_\ast C_n\eta = D_I$.
\end{remark}

Next we show that an interval of the form $(1,x_k\wedge \neg x_{k-1},1)$ and of the form $(1,x_k\iff x_{k-1},1)$ may be obtained from
an interval of the form $(1, x_k \vee \neg x_{k-1},1)$.

\begin{lemma}\label{lem:insert-ands-iffs}
	Let $C = (f_i)$ be a formal $n$-cube and assume that there is some $0\leq k< n$ such that
	\[
    	(f_{k-1}, f_k, f_{k+1}) = (1, x_k\vee\neg x_{k-1}, 1)
	\]
	and $f_i$ is independent of $x_k$ and $x_{k-1}$ for all $i\neq k$.
	Then one can also construct, by gluings and flips, out of $C$
	the cubes $D$ and $E$ with the same label functions as $C$ everywhere except for $i=k$,
	where they are
	\[
	    f_i^D=(x_i\wedge \neg x_{i-1})
    	\quad\text{and}\quad
	    f_i^E=(x_i\iff x_{i-1}),
	\]
    respectively.
\end{lemma}

\begin{proof}
	The occurrence of 1's in the ($k-1$)-th label and in the ($k+1$)-th label means that the only label we have to care about whenever gluing along $x_k$ or $x_{k-1}$ is the label $f_k$.
Now we can glue $C$ to itself along $[x_{k-1}=0]$ to obtain a cube $W$ with label function
	\[f_k^W=x_k.\]
	Gluing $C$ and $W$ along $[x_k=1]$ yields the cube $D = CW$ with label function
	\[f_k^D=(x_k\wedge \neg x_{k-1})\vee (\neg x_k\wedge 0)=x_k\wedge \neg x_{k-1}.\]
    Flipping%
	\footnote{Note that this is one of the rare instances where we need to apply some automorphism.}
    $D$ along $x_{k-1}$ now yields $D'$ with
	\[f_k^{D'}=x_k\wedge x_{k-1},\]
    and gluing $D$ and $D'$ along $[x_k=0]$ yields the cube $E = D'D$ with
	\[f_k^E=(x_k \wedge x_{k-1})\vee (\neg x_k \wedge \neg x_{k-1}).\qedhere\]
\end{proof}

\begin{proof}[Proof of Theorem~\ref{prop:C-D-glue-all}]
Take some non-degenerate degenerable formal $n$-cube $X$ and assume it has the normal form of Corollary~\ref{cor:class-formal-degenerable-non-degenerate-interval}.
Take $I$ to be the set of labels which have label function constant equal to one.
Starting from $C_n$, insert 1's at every place $i\in I$, obtaining $D_I$ from Lemma~\ref{lem:insert_1s}.
In $X$, an expression like $x_i\wedge\neg x_{i-1}$ or $x_i\iff x_{i-1}$ can only appear between 1's,
and so we can apply Lemma~\ref{lem:insert-ands-iffs} to insert those from $D_I$.
\end{proof}

In the rest of this subsection we will prove the following theorem
which is a stronger version of Theorem~\ref{prop:C-D-glue-all}
that also allows for label pushforwards.

\begin{theorem}\label{thm:C-D-glue-all-really}
Every degenerable $n$-cube is generated, under gluing and cube morphisms,
from the cubes $i_\ast C_k$ and $i_\ast D_k$ where $i\colon[k]\to[n]$ runs over all injections.
\end{theorem}

Note that in the theorem, one can also substitute the condition on the simplex maps with the condition on all maps,
but relaxing the generation to only $C_n$.

\begin{proposition}\label{cor:pushforward-generates}
Let $n\in\N$. Then the cubes $\phi_*C_k$ for all $\phi\colon[k]\to[n]$ generate,
under gluing and cube morphisms,
exactly the same cubes as the cubes $i_*C_k$ and $i_*D_k$ for all injections $i\colon[k]\to[n]$.
\end{proposition}

\begin{proof}
For the forward direction note that the face $[x_k=1]$ of $\sigma_{k,\ast}C_{k+1}$ is equal to $D_k$,
where $\sigma_k\colon[k+1]\to[k]$ is the $k$-th degeneracy.
Thus, if $i\colon[k]\to[n]$ is monic, $i_*D_k$ is a face of $i_\ast\sigma_{k,\ast}C_{k+1}$,
where $i\circ\sigma_k\colon[k+1]\to[n]$.
So $i_\ast C_k$ and $i_\ast D_k$ for $i\colon[k]\to[n]$ injective can be generated by $\phi_\ast C_k$ for all $\phi$.

For the backwards direction, let $e\colon[k]\to[n]$ be surjective.
Then $e_\ast C_k$ is equal to a blown up $D_I$ for some $I\subset[n]$ by Remark~\ref{rem:insert-1s-pushforward}.
By Lemma~\ref{lem:insert_1s}, $D_I$ is generated by $C_n$ and $D_n$ under gluing and flips,
and therefore, so is $e_\ast C_k$.
For general $\phi\colon[k]\to[n]$ take its mono-epi factorization $\phi=i\circ e$
and then use that we can glue as before the cubes $i_\ast C_m$ and $i_\ast D_m$ to obtain $i_\ast e_\ast C_k = \phi_\ast C_k$,
see Remark~\ref{rem:formal-gluing-push}.
\end{proof}

In order to prove Theorem~\ref{thm:C-D-glue-all-really} we introduce another kind of cube
which will be also of great importance later on.

\begin{definition}\label{def:C-D-push}
Define $C_n^i$ and $D_n^i$ for $0\le i\le n+1$ by pushforward of $C_n$, resp. $D_n$, with the unique injective map $\delta_i\colon[n]\to[n+1]$ that misses $i$.
\end{definition}

\begin{remark}\label{rem:C-D-push}
In formulas, $C_n^i$ is the tuple $C_n$ with a 0 inserted as $i$-th coordinate, and similarly for $D_n^i$:
\[
    f_j^{C_n^i} = \begin{cases}
        x_j\vee\neg x_{j-1},\quad & 0\leq j\leq i-1, \\
        0,\quad & j = i, \\
        x_{j-1}\vee\neg x_{j-2},\quad & i+1\leq j\leq n+1.
    \end{cases}
\quad\quad
    f_j^{D_n^i} = \begin{cases}
        f_j^{D_n},\quad & 0\leq j\leq i-1, \\
        0,\quad & j = i, \\
        f_{j-1}^{D_n},\quad & i+1\leq j\leq n+1.
    \end{cases}
\]
\end{remark}

\begin{example}[$n=3$]
In pictures:
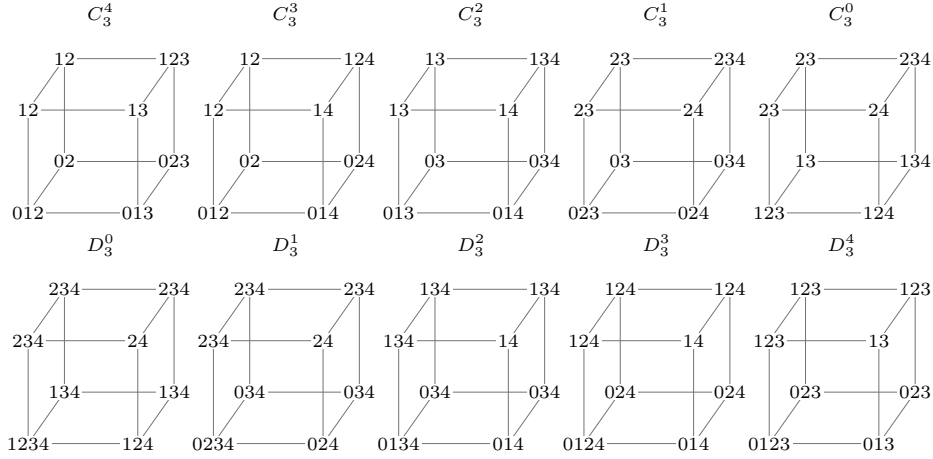
\begin{center}
\begin{tikzpicture}

\begin{scope}[shift={(0,3.05)}]
  \smallcube{C_3^4}{12}{123}{12}{13}{02}{023}{012}{013}
\end{scope}

\begin{scope}[shift={(2.45,3.05)}]
  \smallcube{C_3^3}{12}{124}{12}{14}{02}{024}{012}{014}
\end{scope}

\begin{scope}[shift={(4.90,3.05)}]
  \smallcube{C_3^2}{13}{134}{13}{14}{03}{034}{013}{014}
\end{scope}

\begin{scope}[shift={(7.35,3.05)}]
  \smallcube{C_3^1}{23}{234}{23}{24}{03}{034}{023}{024}
\end{scope}

\begin{scope}[shift={(9.80,3.05)}]
  \smallcube{C_3^0}{23}{234}{23}{24}{13}{134}{123}{124}
\end{scope}

\begin{scope}[shift={(0,0)}]
  \smallcube{D_3^0}{234}{234}{234}{24}{134}{134}{1234}{124}
\end{scope}

\begin{scope}[shift={(2.45,0)}]
  \smallcube{D_3^1}{234}{234}{234}{24}{034}{034}{0234}{024}
\end{scope}

\begin{scope}[shift={(4.90,0)}]
  \smallcube{D_3^2}{134}{134}{134}{14}{034}{034}{0134}{014}
\end{scope}

\begin{scope}[shift={(7.35,0)}]
  \smallcube{D_3^3}{124}{124}{124}{14}{024}{024}{0124}{014}
\end{scope}

\begin{scope}[shift={(9.80,0)}]
  \smallcube{D_3^4}{123}{123}{123}{13}{023}{023}{0123}{013}
\end{scope}
\end{tikzpicture}
\end{center}
\end{example}

The special property of the $C_n^i$ is that one can iteratively glue them together,
obtaining the gluings $V_n^k$.

\begin{lemma}\label{lem:glue_C}
Let $V_n^0=C_n^0$ and then inductively define $V_n^k$ for $k\leq n$ as the gluing of
$C_n^{k}$ and $V_n^{k-1}$ along the face $[x_{k-1}=0]$.
It fulfills the formula
\[
    f_j^{V_n^k}=
    \begin{cases}
        x_j\vee\neg x_{j-1},\quad & 0\le j\le k-1, \\
        \neg x_{k-1} &j=k, \\
        x_k& j=k+1, \\
        x_{j-1}\vee\neg x_{j-2}&k+2\le j\le n+1.
    \end{cases}
\]
\end{lemma}

\begin{proof}
	Induct on $k$. For $k=0$ the formula is precisely the one of $C_n^0$.
	Recall the formula for $C_n^{k+1}$ from Remark~\ref{rem:C-D-push}.
	Outside of the labels $j\in\{k,k+1,k+2\}$ this agrees with $V_n^k$,
	and is independent of $x_k$ there.
	So when gluing along $[x_k=0]$ we can focus on these three labels.
	There, we have that
	\begin{align*}
	(f_k^{V_n^k},f_{k+1}^{V_n^k},f_{k+2}^{V_n^k})=(&\neg x_{k-1},&&x_k,&&&  x_{k+1}\vee\neg x_{k})\\
	(f_k^{C_n^{k+1}}, f_{k+1}^{C_n^{k+1}},f_{k+2}^{C_n^{k+1}})=(&x_k\vee\neg x_{k-1},&&0,&&& x_{k+1}\vee\neg x_k)
	\end{align*}
	and hence evaluating at $x_k=0$ both yield the triple $(\neg x_{k-1},0,1)$.
	Thus we may glue them along $[x_k=0]$ and obtain
    \[
	V_n^{k+1} = C_n^{k+1}V_n^k = (x_k\wedge C_n^{k+1}\rvert_{x_k=1})\vee (\neg x_k\wedge V_n^k\rvert_{x_k=1}).
    \]
	In the relevant labels this computes to
	\begin{align*}
	&(x_k\wedge (1,0,x_{k+1}))\vee (\neg x_k \wedge (\neg x_{k-1},1,x_{k+1})) \\
    = &(x_k\vee (\neg x_k\wedge \neg x_{k-1}), \neg x_k, (x_k\wedge x_{k+1})\vee (\neg x_k\vee x_{k+1}))\\
	= &(x_k\vee \neg x_{k-1}, \neg x_k, x_{k+1})
		\end{align*}
which completes the induction and yields the desired formula.
\end{proof}

\begin{remark}\label{rem:glue-C-inverse}
Since gluing is (up to flips) self-inverse and $V_n^n = C_n^{n+1}$, the gluing $\prod_{j>k}C_n^j$, starting from $C_n^{n+1}$, yields the same result, i.e.,
\[\prod_{j=n+1}^{k+1}C_n^j=\prod_{j=0}^k C_n^j.\]
In particular, omitting flips notationally,
\[C_n^k=(\prod_{j<k}C_n^j)(\prod_{j>k} C_n^j).\]
\end{remark}
A similar gluing procedure is possible with the cubes $D_n^i$ as well.

\begin{lemma}\label{lem:glue_D}
	For $0\le k\le n$, one can iteratively glue $D_n^{k}$ along the face $[x_{k-1}=0]$
	to the gluing $\prod_{j\le k-1}D_n^j$ and obtains a cube with label functions
	\[f_j=\begin{cases}
			  x_j\vee \neg x_{j-1} &0\le j\le k-1\\
			  \neg x_{k-1} & j=k\\
			  x_k & j=k+1\\
			  x_{j-1}\vee \neg x_{j-2} & k+2\le j\le n\\
			  1&j=n+1
	\end{cases}\]
\end{lemma}

\begin{proof}
Completely analogous to the proof of Lemma~\ref{lem:glue_C} where we glued the $C_n^k$.
Another possibility is to note that the face $[x_n=1]$ of $\sigma_{n,\ast}C_{n+1}$ is $D_n$
and hence the face $[x_n=1]$ of $\sigma_{n+1,\ast}C_{n+1}^k$ is $D_n^k$ for $k\le n$.
Thus we could deduce the lemma also directly from Lemma~\ref{lem:glue_C}.
Note that one cannot glue $D_n^{n+1}$ to $\prod_{j=0}^n D_n^j$.
\end{proof}

\begin{proof}[Proof of Theorem~\ref{thm:C-D-glue-all-really}]
Consider a degenerable $n$-cube $X = (f_i)_{i=0,\ldots,n}$ which we can assume to be of the normal form~\ref{prop:class-formal-degenerable}.
We show that it can be obtained by a (finite) sequence of gluing (pullbacks of) the formal cubes $i_\ast C_l$ and $i_\ast D_l$,
where $i\colon[l]\to[n]$ runs over all injections.
We do so by reducing the problem step by step,
where in each step $k$ we show that the degenerable cube $X_{k-1}$ (starting from $X = X_{-1}$)
can be constructed from $C_n$, $D_n$ and a degenerable(!) cube $X_k$ which is considerably easier than its predecessor,
reducing the problem to the trivial case for $k$ big enough.
\begin{enumerate}
    \item[(0)] In a first step, we get rid of all label functions that are constant zero.
        Denote by $\phi\colon[m]\to[n]$ the unique injective map in $\bbDelta$
        whose image consists of all $i$ such that $f_i$ is not constant $0$ (if $\phi = \id$, skip this step).
        Consider the $m$-cube $X_0$ with labels in $[m]$ which is given by $(f_{\phi(j)})_{j=0,\ldots,m}$
        but pulled back along the corresponding cube morphism $\gamma\colon\{0,1\}^m\to\{0,1\}^n$ such that
        $f_{\phi(j)}$ depends on $x_j$ and $x_{j-1}$, and not on $x_{\phi(j)}$ and $x_{\phi(j)-1}$.
        This makes $X_0$ a degenerable $m$-cube by Proposition~\ref{prop:class-formal-degenerable},
        and $X$ can be constructed from $X_0$:
        If $f_i = 0$, then both $f_{i-1}$ and $f_{i+1}$ are independent of $x_i$ and $x_{i-1}$ by Proposition~\ref{prop:class-formal-degenerable}.
        Thus, for $\phi_\ast X_0$, we can apply an appropriate cube morphism $\eta\colon\{0,1\}^n\to\{0,1\}^m$ (which is a retract to $\gamma$)
        such that $\eta^\ast(\phi_\ast X_0) = X$, since $\phi_\ast X_0$ has label functions $f_i$ if $\phi(j) = i$ and $0$ otherwise.
        Therefore, we can assume in the following that in $X$, no label functions equal to $0$ occur,
        which also applies to $f_n = x_n = 0$ by our convention.
    \item[(1)] In the second step, we get rid of all occurrences of \enquote{ands} and \enquote{iffs}.
        If $f_i = x_i\wedge\neg x_{i-1}$ or $f_i = x_i\iff x_{i-1}$,
        then by Proposition~\ref{prop:class-formal-degenerable} we know that $f_{i-1} = 1$ and $f_{i+1} = 1$.
        Thus, take the $n$-cube $X_1$ with labels in $[n]$ where all occurrences of
        \enquote{ands} and \enquote{iffs} have been replaced by the expression $x_i\vee\neg x_{i-1}$,
        i.e., $X_1$ has label functions
        \[
            g_i = \begin{cases}
                x_i\vee\neg x_{i-1},\quad & f_i\in\{x_i\wedge\neg x_{i-1},x_i\iff x_{i-1}\}, \\
                f_i,\quad & \text{else}.
            \end{cases}
        \]
        Again by Proposition~\ref{prop:class-formal-degenerable}, $X_1$ is degenerable.
        Lemma~\ref{lem:insert-ands-iffs} shows that $X$ can be constructed from $X_1$.
        Thus we can additionally assume that $X$ has no occurrence of \enquote{ands} or \enquote{iffs}.
    \item[(2)] Next, we get rid of all occurrences of label functions of the form $f_i = \neg x_{i-1}$ which are not followed by the label function $f_{i+1} = x_{i+1}$.
        In this case, $f_{i+1} = 1$ (since $0$ can't occur by assumption of step 0).
        Thus take the $n$-cube $X_2$ with labels in $[n]$ where all such $f_i$ are replaced with $x_i\vee\neg x_{i-1}$.
        This cube $X_2$ is degenerable by Proposition~\ref{prop:class-formal-degenerable}.
        Moreover, pulling $X_2$ back along the endomorphism $\{0,1\}^n\to\{0,1\}^n$ of cubes which evaluates $x_i = 0$ for each $i$ where we have replaced $f_i$ yields $X$.
    \item[(3)] The same argument applies to all occurrences of label functions of the form $f_i = x_i$ which are not preceded by $f_{i-1} = \neg x_{i-2}$.
        But beware that one has to restrict to the face $x_{i-1} = 1$.
        Together we can assume that $X$ has as label functions only $1$, $x_i\vee\neg x_{i-1}$, $x_i$ or $\neg x_{i-1}$ but with the restriction that in the latter two cases, there have to appear sequences of the form
        \[
            (f_{i-1},f_i) = (\neg x_{i-2},x_i)\quad\text{or}\quad (f_i,f_{i+1}) = (\neg x_{i-1},x_{i+1}).
        \]
    \item[(4) and (5)] In the last step we construct the degenerable cube $X$ from left to right,
        inserting ones and pairs of the form $(\neg x_{i-1},x_{i+1})$.

        We now inductively build $X = (f_i)_i$ from $i_\ast C_p$ and $i_\ast D_p$ under gluings for injections $i\colon[p]\to[n]$.
        By Remark~\ref{rem:formal-gluing-push}, it suffices to show that we can build $X$ from
        $i_\ast C_p$ and $i_\ast D_p$ for injections $i\colon[p]\to[q]$ where $q\leq n$.
        Let $m = n - k$ where $k$ is the number of label functions of the form $f_i = \neg x_{i-1}$
        (where $f_n$ doesn't count since there $\neg x_{n-1} = x_n\vee\neg x_{n-1}$).
        
        The induction start is as follows:
        let $i_0$ be the smallest index where $f_i$ differs from $x_i\vee\neg x_{i-1}$.
        We distinguish two cases:
        \begin{enumerate}[(i)]
            \item We have that $f_{i_0} = 1$.
                By Lemma~\ref{lem:insert_1s}, the $m$-cube $Y = (g_i)_i$ with label functions in $[m]$
                given by $g_i = x_i\vee\neg x_{i-1}$ for $i\neq i_0$ and $g_{i_0} = 1$,
                is glued from $C_m$ and $D_m$, and thus is in the gluing closure (and is degenerable).
                Note that the $m$-cube $Y'$ with the same label functions as $Y$ but $g'_m = 1$
                also can be glued by Lemma~\ref{lem:insert_1s}.
            \item We have that $f_{i_0} = \neg x_{i_0-1}$.
                By Lemma~\ref{lem:glue_C}, the ($m+1$)-cube $Y = (g_i)_i$ with label functions in $[m+1]$
                given by $g_i = x_i\vee\neg x_{i-1}$ for $i\neq i_0,i_0+1$ and
                $(g_{i_0},g_{i_0+1}) = (\neg x_{i_0-1},x_{i_0+1})$ can be generated:
                it is $\gamma^\ast V_m^{i_0}$ for an appropriate morphism of cubes $\gamma$
                that shifts $x_{i_0}$ to $x_{i_0+1}$, making $Y$ degenerable.
                Note that the ($m+1$)-cube $Y'$ with the same label functions as $Y$ but $g'_{m+1} = 1$
                also can be glued by Lemma~\ref{lem:glue_D}.
        \end{enumerate}
        In the induction step, suppose that we have constructed a degenerable $l$-cube $Y = (g_i)_i$ with labels in $[l]$ such that there exists $i_0$ such that for all $i < i_0$ we have that $f_i = g_i$, $f_{i_0}\ne g_{i_0}$ and
        $i_0$ is minimal w.r.t. this property.
        Suppose also the existence of the degenerable $l$-cube $Y'$ with labels in $[l]$ which agrees with $Y$ except for $g'_l = 1$.
        We increase $i_0$ by case distinction.
        \begin{enumerate}[(i)]
            \item If $f_{i_0} = 1$, we glue the $l$-cube $Z = (h_i)_i$ with labels in $[l]$
                that satisfy $h_i = g_i$ for $i\neq i_0$ and $h_{i_0} = 1$ from $Y$ and $Y'$,
                by Lemma~\ref{lem:insert_1s}:
                the only label functions that matter in the gluing are the ones at positions
                $i \geq i_0$, and there we have that $g_i = x_i\vee\neg x_{i-1}$ and $g'_l = 1$.
                Also, we obtain the $l$-cube $Z' = (h'_i)_i$ with $h'_i = h_i$ for $i<l$ and $h'_l = 1$.
                Both cubes $Z$ and $Z'$ are degenerable, and the minimal $i_1$ such that  $f_{i_1} \ne h_{i_1}$ satisfies $i_0 < i_1$.
            \item If $f_{i_0} = \neg x_{i_0-1}$, we can construct the ($l+1$)-cube $Z = (h_i)_i$
                with labels in $[l+1]$ satisfying
                \[
                    h_i = \begin{cases}
                        g_i,\quad & i<i_0, \\
                        \neg x_{i-1},\quad & i = i_0, \\
                        x_i,\quad & i = i_0+1, \\
                        x_i\vee\neg x_{i-1},\quad & i\geq i_0+2
                    \end{cases}
                \]
                by Lemma~\ref{lem:glue_C}
                (and the same cube $Z'$ but with $h'_{l+1} = 1$ by Lemma~\ref{lem:glue_D}).
                Indeed, $Z$ can be glued from $\delta_{i,\ast} Y$ for $i\geq i_0$,
                since in Lemma~\ref{lem:glue_C}, by Remark~\ref{rem:glue-C-inverse},
                the gluing for $V_l^{i_0}$ only depends on cube faces $[x_i = 0]$ for $i\geq i_0$
                and happens only in labels $i\geq i_0$, and there, $\delta_{i,\ast} Y$ agrees with $C_l^i$.
                Thus, this gluing procedure doesn't affect label functions $h_i$ for $i<i_0$.
                To make $Z$ degenerable, we afterwards apply a cube morphism to shift the coordinates
                one step higher, starting from $x_{i_0}$ (which is mapped to $x_{i_0+1}$).
                A similar argument works for $Z'$.
                There, one has to be careful about gluing from right to left, starting from $D_k^k$,
                since $\prod_{j=0}^l D_l^j$ is not $D_l^{l+1}$, but has as label functions
                \[
                    f_j = \begin{cases}
                        x_j\vee\neg x_{j-1},\quad & 0\leq j\leq l-1, \\
                        \neg x_{j-1},\quad & j = l, \\
                        1,\quad & j = l+1.
                    \end{cases}
                \]
                But this cube is just $D_{l+1}$ evaluated at $x_l = 0$, and therefore also exists.
                Thus, we can also apply Lemma~\ref{lem:glue_D} in the reverse gluing direction.
                Again, the minimal $i_1$ such that for all $i < i_1$ we have that $f_i = h_i$ satisfies $i_0 < i_1$.
        \end{enumerate}
        After a finite iteration, this procedure terminates (as soon as $l = n$ and there is no $i_0$ anymore with $f_{i_0}\neq g_{i_0}$)
        and the resulting cube is equal to $X$,
        since in each step where $f_{i_0} = \neg x_{i_0-1}$, the dimension of the cube is increased,
        and $k$ was chosen so that this dimension increase happens exactly $k$ times,
        therefore $m + k = n$.\qedhere
\end{enumerate}\end{proof}

We also have the reverse direction of Theorem~\ref{thm:C-D-glue-all-really}.

\begin{proposition}\label{prop:degenerable-glue-C-D}
Every $k$-cube $\phi_\ast C_k$ and $\phi_\ast D_k$ with labels in $[n]$ for some map $\phi\colon[k]\to[n]$ is generated, under gluing and cube morphisms, from the degenerable $n$-cubes.
\end{proposition}

\begin{proof}
By Proposition~\ref{cor:pushforward-generates}, it suffices to show that the formal cubes $i_\ast C_k$ and $i_\ast D_k$ for $i\colon[k]\to[n]$ are generated under gluing from pullbacks of degenerable $n$-cubes.
But then, since pushforward of labels commutes with gluing by Remark~\ref{rem:formal-gluing-push},
it suffices to show it for $\delta_{i,\ast} C_k = C_k^i$ and $\delta_{i,\ast} D_k = D_k^i$.
But this directly follows from Lemmas~\ref{lem:glue_C} and~\ref{lem:glue_D}:
the coordinate-shifted version of $V_k^i$ is degenerable, and therefore after shifting coordinates back,
we can generate all $V_k^i$ with cube morphisms, and $C_k^i = V_k^{i-1}V_k^i$ (ignoring flips).
Similarly for $D_k^i$.
\end{proof}

\subsection{Generating Faces of \texorpdfstring{$C_{n+1}$}{C(n+1)}}\label{ssec:generating_faces}

In this subsection we show that every outer face of $C_{n+1}$ can be glued from the $C_n^i = \delta_{i,\ast}C_n$ and $D_n^i = \delta_{i,\ast}D_n$.

\begin{definition}
We denote by $C_{n+1}^{x_i=a} = C_{n+1}(x_i=a)$ the face $[x_i=a]$ of $C_{n+1}$.
\end{definition}

\begin{remark}
The label functions of $C_{n+1}^{x_i=a}$ are given by the ones of $C_n$ except at positions $i$ and $i+1$: if $x_i = 1$ (the upper face), then $(f_i,f_{i+1}) = (1,x_i)$ and if $x_i = 0$ (the lower face), then $(f_i,f_{i+1}) = (\neg x_{i-1},1)$.
Note the shift in the cube dimension.
\end{remark}
    
\begin{example}
The eight faces of
\[
    C_4=(x_0\vee \neg x_{-1}, x_1\vee \neg x_0, x_2\vee\neg x_1, x_3\vee\neg x_2,x_4\vee\neg x_3)
\]
(with $x_{-1}=1$ and $x_4=0$ as always) are
\begin{align*}
C_4^{x_0=1}
  &= (1,\ &&x_0,\ &&&x_1\vee \neg x_0,\ &&&&x_2\vee \neg x_1,\ &&&&&x_3\vee \neg x_2),\\
C_4^{x_0=0}
  &= (\neg x_{-1},\ &&1,\ &&&x_1\vee \neg x_0,\ &&&&x_2\vee \neg x_1,\ &&&&&x_3\vee \neg x_2),\\
C_4^{x_1=1}
  &= (x_0\vee \neg x_{-1},\ &&1,\ &&&x_1,\ &&&&x_2\vee \neg x_1,\ &&&&&x_3\vee \neg x_2),\\
C_4^{x_1=0}
  &= (x_0\vee \neg x_{-1},\ &&\neg x_0,\ &&&1,\ &&&&x_2\vee \neg x_1,\ &&&&&x_3\vee \neg x_2),\\
C_4^{x_2=1}
  &= (x_0\vee \neg x_{-1},\ &&x_1\vee \neg x_0,\ &&&1,\ &&&&x_2,\ &&&&&x_3\vee \neg x_2),\\
C_4^{x_2=0}
  &= (x_0\vee \neg x_{-1},\ &&x_1\vee \neg x_0,\ &&&\neg x_1,\ &&&&1,\ &&&&&x_3\vee \neg x_2),\\
C_4^{x_3=1}
  &= (x_0\vee \neg x_{-1},\ &&x_1\vee \neg x_0,\ &&&x_2\vee \neg x_1,\ &&&&1,\ &&&&&x_3),\\
C_4^{x_3=0}
  &= (x_0\vee \neg x_{-1},\ &&x_1\vee \neg x_0,\ &&&x_2\vee \neg x_1,\ &&&&\neg x_2,\ &&&&&1).
\end{align*}
and as pictures
are given by the eight 3-cubes
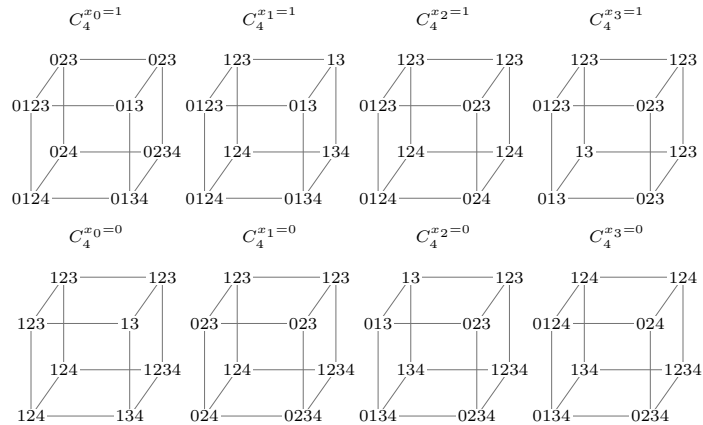
\begin{center}
\begin{tikzpicture}[scale=.9, every node/.style={transform shape}]

\begin{scope}[shift={(0,3.2)}]
  \smallcube{C_4^{x_0=1}}{023}{023}{0123}{013}{024}{0234}{0124}{0134}
\end{scope}

\begin{scope}[shift={(2.55,3.2)}]
  \smallcube{C_4^{x_1=1}}{123}{13}{0123}{013}{124}{134}{0124}{0134}
\end{scope}

\begin{scope}[shift={(5.10,3.2)}]
  \smallcube{C_4^{x_2=1}}{123}{123}{0123}{023}{124}{124}{0124}{024}
\end{scope}

\begin{scope}[shift={(7.65,3.2)}]
  \smallcube{C_4^{x_3=1}}{123}{123}{0123}{023}{13}{123}{013}{023}
\end{scope}

\begin{scope}[shift={(0,0)}]
  \smallcube{C_4^{x_0=0}}{123}{123}{123}{13}{124}{1234}{124}{134}
\end{scope}

\begin{scope}[shift={(2.55,0)}]
  \smallcube{C_4^{x_1=0}}{123}{123}{023}{023}{124}{1234}{024}{0234}
\end{scope}

\begin{scope}[shift={(5.10,0)}]
  \smallcube{C_4^{x_2=0}}{13}{123}{013}{023}{134}{1234}{0134}{0234}
\end{scope}

\begin{scope}[shift={(7.65,0)}]
  \smallcube{C_4^{x_3=0}}{124}{124}{0124}{024}{134}{1234}{0134}{0234}
\end{scope}
\end{tikzpicture}
\end{center}
\end{example}

The goal of this subsection is to prove the following theorem.

\begin{theorem}\label{thm:faces_of_Cn_generated}
The faces of $C_{n+1}$ are generated under gluing by $C_n^i$ and $D_n^i$ for $i\le n+1$.
\end{theorem}

There are two easy faces.

\begin{remark}\label{rem:C-faces-simple}
Lemma~\ref{lem:glue_D} directly implies that the faces
\[
    C_{n+1}^{x_n=1}=D_n^{n+1}\quad\text{and}\quad C_{n+1}^{x_n=0}=\prod_{j\le n} D_n^j
\]
can be glued from the $D_n^i$'s.
\end{remark}

From there, we approach the construction of the other faces of $C_{n+1}$.
We do so by observing that the $i$-th and ($i+1$)-th upper, resp. lower, faces have a common intersection, the \emph{($i+1$)-th diagonal $S_{n+1}^{i+1}$},
which is a diagonal face of the $n+1$-cube, namely $[1-x_i = x_{i+1}]$.

\begin{proposition}
Let $X$ be a formal $n$-cube with labels in $[m]$.
Then for each $0\le i<n-1$, the formal ($n-1$)-cubes $X^{x_i=0}=X(x_i=0)$ and $X(x_{i+1}=0)$ agree on the face $[x_i=0]$,
as do the ($n-1$)-cubes $X(x_i=1)$ and $X(x_{i+1}=1)$ on the face $[x_i=1]$.
Moreover, their gluing agrees and is equal to the diagonal $[1-x_i = x_{i+1}]$, in formulas
\[
    X(x_{i+1}=0)X(x_i=0) = X(x_i=1)X(x_{i+1}=1) = X(1-x_i = x_{i+1}).
\]
\end{proposition}

\begin{proof}
Denote by $\epsilon_i^{n-1}\colon\{0,1\}^{n-1}\to\{0,1\}^n$ the inert inclusion onto the face $[x_i = 0]\subset\{0,1\}^n$.
By~\cite[Remark 3.1.2]{Bihlmaier2026} we have that $\epsilon_{i+1}^{n-1}\circ\epsilon_i^{n-2} = \epsilon_i^{n-1}\circ\epsilon_i^{n-2}$
which means that $X(x_i=0)$ and $X(x_{i+1}=0)$ agree on the face $[x_i=0]$.
The same holds for $0$ replaced by $1$.

Denote by $A$ the gluing $X(x_{i+1}=0)X(x_i=0)$,
by $B$ the gluing $X(x_i=1)X(x_{i+1}=1)$ and by $C$ the cube $X(1-x_i = x_{i+1})$.
Then, if $x_i = 0$, we have by the definition of formal gluing~\ref{def:formal-gluing} that
\begin{align*}
    A(x) &= X(x_i=0,x_{i+1}=1), \\
    B(x) &= X(x_{i+1}=1,x_i=0), \\
    C(x) &= X(x_i=0,1-x_i=x_{i+1}) = X(x_i=0,x_{i+1}=1)
\end{align*}
and for $x_i = 1$ we obtain
\begin{align*}
    A(x) &= X(x_{i+1}=0,x_i=1), \\
    B(x) &= X(x_i=1,x_{i+1}=0), \\
    C(x) &= X(x_i=1,1-x_i=x_{i+1}) = X(x_i=1,x_{i+1}=0).
\end{align*}
A close inspection yields that $A = B = C$.
\end{proof}

\begin{corollary}\label{cor:diagonals-suffice}
Let $X$ be a formal $n$-cube with labels in $[m]$.
Then each lower face $[x_i=0]$ of $X$ can be glued from the lower face $X(x_{n-1}=0)$ and the diagonals $X(1-x_i=x_{i+1})$.
Analogously, every upper face $[x_i=1]$ of $X$ can be glued from the upper face $X(x_{n-1}=1)$ and the diagonals $X(1-x_i=x_{i+1})$.
\end{corollary}

In the case of $X=C_{n+1}$, the diagonals $C_{n+1}(1-x_{i-1}=x_i) = S_{n+1}^i$ take the following form
\[
    f_j^{S_{n+1}^{i}} =
	\begin{cases}
	    x_j\vee\neg x_{j-1},\quad &0\le j< i, \\
	    \neg x_{i-1},\quad &j=i, \\
        x_{i-1}\vee x_{i},\quad &j=i+1\quad\text{(no negation!)}, \\
	    x_{j-1}\vee\neg x_{j-2}&i+2\le j\le n+1.
    \end{cases}
\]
for $0<i\leq n$.

\begin{example}
For $n=3$, the diagonals look as follows.
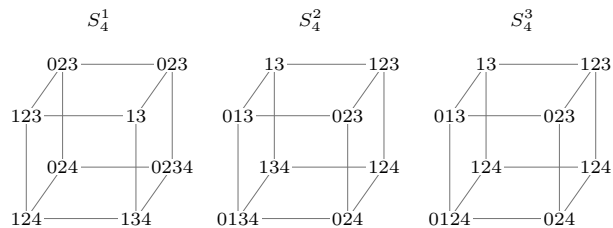
\begin{center}
\begin{tikzpicture}
\begin{scope}[shift={(0,0)}]
  \smallcube{S_4^1}{023}{023}{123}{13}{024}{0234}{124}{134}
\end{scope}

\begin{scope}[shift={(2.8,0)}]
  \smallcube{S_4^2}{13}{123}{013}{023}{134}{124}{0134}{024}
\end{scope}

\begin{scope}[shift={(5.6,0)}]
  \smallcube{S_4^3}{13}{123}{013}{023}{124}{124}{0124}{024}
\end{scope}
\end{tikzpicture}
\end{center}
Note that $S_{n+1}^i$ only contains one ($n+1$)-tuple, which is the same as the one from $D_n^{i}$.
\end{example}

To arrive at the theorem, we thus need to construct the diagonals $S_{n+1}^k$.
The first ingredient is a surprising iterated composition of $C_n^k$ with $D_n^k$.
By their inductive definition, it is easy to see that $C_n^k$ and $D_n^k$ share one common side.
However, we are able to continue composing with the same $C_n^k$ in different directions.

\begin{lemma}\label{lem:gluing-W}
Let $0<k\leq n$ and define the formal cube $W_n^k$ by the formula
\[
    f_j^{W_n^k}=\begin{cases}
    x_j\vee\neg x_{j-1},\quad & 0\le j\le k-1, \\
    0,\quad & j=k, \\
    1,\quad &j=k+1, \\
    x_{j-1}\vee\neg x_{j-2},\quad & k+2\le j\le n+1.
    \end{cases}
\]
Then $W_n^k$ is generated under gluing by $D_n^k$ and $C_n^k$,
where $D_n^k$ is used once and $C_n^k$ is used ($n-k$)-times.
\end{lemma}

\begin{proof}
Let $0<k\leq n$. By Lemma~\ref{lem:insert_1s}, the cube $D_{\{k\}}$ is generated under gluing by $C_n$ and $D_n$.
Now $W_n^k = \delta_{k,\ast} D_{\{k\}}$ and so by Remark~\ref{rem:formal-gluing-push}, $W_n^k$ is generated under gluing by $C_n^k = \delta_{k,\ast}C_n$ and $D_n^k = \delta_{k,\ast}D_n$.
The last assertion follows from Remark~\ref{rem:insert_1s}.
\end{proof}

Recall the cubes $V_n^k$ from Lemma~\ref{lem:glue_C} which are obtained by an iterated gluing of the $C_n^i$.

\begin{lemma}\label{lem:gluing-S}
Let $0<k\le n$.
Gluing $W_n^k$ and $V_n^{k-1}$ along the face $[x_{k-1}=0]$ yields $S_{n+1}^k$.
\end{lemma}

\begin{proof}
Recall the label functions of $W_n^k$ and $V_n^{k-1}$.
They only differ from each other (and from the desired result) in labels $k-1$, $k$ and $k+1$ and contain no $x_{k-1}$ outside of those, so we focus on them.
They are given by
\begin{align*}
	W_n^k=(&x_{k-1}\vee\neg x_{k-2}, &&0,&&&1)\\
    V_n^{k-1}=(&\neg x_{k-2},&&x_{k-1},&&&x_{k}\vee\neg x_{k-1})\\
\end{align*}
Plugging in $x_{k-1}=0$ yields $(\neg x_{k-2}, 0,1)$ for both and thus we may glue, resulting in
\[
    W_n^k V_{n}^{k-1}=(x_{k-1}\wedge W_n^k\rvert_{x_{k-1}=1})\vee(\neg x_{k-1}\wedge V_n^{k-1}\rvert_{x_{k-1}=1}),
\]
which in the relevant coordinates computes to
\begin{align*}
	(x_{k-1}\wedge (1,0,1))\vee (\neg x_{k-1}\wedge (\neg x_{k-2}, 1, x_k))
    &=(x_{k-1}\vee (\neg x_{k-1}\wedge \neg x_{k-2}), \neg x_{k-1}, x_{k-1}\vee(\neg x_{k-1}\wedge x_k))\\
	&=(x_{k-1}\vee\neg x_{k-2},\neg x_{k-1}, x_{k}\vee x_{k-1}).
\end{align*}
So we obtain the label functions of $S_{n+1}^k$.
\end{proof}

This concludes the proof of Theorem~\ref{thm:faces_of_Cn_generated}, let us summarize how.

\begin{proof}[Proof of Theorem~\ref{thm:faces_of_Cn_generated}]
By Remark~\ref{rem:C-faces-simple}, the faces $C_{n+1}^{x_n=1}$ and $C_{n+1}^{x_n=0}$ are generated under gluing by $C_n^i$ and $D_n^i$.
Corollary~\ref{cor:diagonals-suffice} shows that it thus suffices to generate all diagonals $S_{n+1}^k$
under gluing.
By Lemma~\ref{lem:glue_C}, the cubes $V_n^i$ and by Lemma~\ref{lem:gluing-W}, the cubes $W_n^i$ are glued from the $C_n^i$ and $D_n^i$.
Therefore, by Lemma~\ref{lem:gluing-S}, also the diagonals are glued from $C_n^i$ and $D_n^i$.
\end{proof}

The proof of the theorem also implies the following.

\begin{corollary}\label{rem:D-faces}
The faces $D_{n+1}^{x_k=0}$ of $D_{n+1}$ for $0\leq k\leq n$ are generated under gluing by $D_n^i$ for $0\leq i\leq n+1$.
\end{corollary}

\begin{proof}
The cube $D_{n+1}$ is the same as $C_{n+1}$ but with a $1$ for the last label
and the same holds for all $D_n^i$ with $0\leq i\leq n$.
Note that in generating the faces $C_{n+1}^{x_i=0}$ we only used $D_n^i$ and $C_n^i$ for $i\leq n$
and the gluings only affected labels $\leq n$.
Therefore, the same argument holds true for $D_{n+1}^{x_k=0}$ substituting each appearance of $C_n^i$ with $D_n^i$.
\end{proof}

\begin{remark}
Note that for the faces of $C_{n+1}$, we used $D_{n}^{n+1}$ exactly once,
namely to show existence of $C_{n+1}^{x_n=1}$.
Indeed, one can't obtain any face $D_{n+1}^{x_k=1}$ by gluings of $C_n^i$'s or $D_n^i$'s,
since at the point $(1,\ldots,1)$, the full value $\{0,\dots,n+1\}$ occurs in $D_{n+1}^{x_k=1}$.
\end{remark}

\section{Proof of the Main Theorem}

In this section, we define the functor $K^*\colon\cSet\to\sSet$ and prove our main result.

\begin{definition}\label{def:kernel}
Define a functor $K\colon{\bbox}^\op\times\bbDelta\to\set$ 
as the subfunctor of the functor
\begin{center}
\begin{tikzcd}
	{{\bbox}^\op\times\bbDelta} & {\set^\op\times\set} & \set
	\arrow["{?\times\su}", from=1-1, to=1-2]
	\arrow["\hom", from=1-2, to=1-3]
\end{tikzcd}
\end{center}
generated by the assignment
\[
(\langle m\rangle,[n])\mapsto
\begin{cases}
    \{A\colon\{0,1\}^n\to\su[n]\mid A\;\text{degenerable}\},\quad &\text{if } m = n, \\
    \emptyset,\quad &\text{if } m\neq n.
\end{cases}
\]
\end{definition}

\begin{lemma}
Concretely, $K$ is given on objects $\langle m\rangle\in\bbox$, $[n]\in\bbDelta$ as
\begin{align*}
    K_n(m)
    &= \{ \sigma_*\circ A\circ\phi\colon \{0,1\}^m\to\su[n]\mid\,
            \phi\in\hom_{\bbox}(\langle m\rangle,\langle k\rangle),\,
            \sigma\in\hom_{\bbDelta}([k],[n]),\,
            A\;\text{degenerable}\} \\
    &\sub \hom(\{0,1\}^m, \su[n])
\end{align*}
and on morphisms by pullback and (double) pushforward.
\end{lemma}

\begin{proof}
Clearly, this defines a functor ${\bbox}^\op\times\bbDelta\to\set$ which is a subfunctor of the described functor,
and it certainly is the smallest.
\end{proof}

So $K_n(m)$ consists of all maps $\{0,1\}^m\to\su[n]$ that are obtainable as a composition
\begin{center}
\begin{tikzcd}[cramped]
	{\{0,1\}^m} & {\{0,1\}^k} & {\su[k]} & {\su[n]}
	\arrow["{\phi}", from=1-1, to=1-2]
	\arrow["A", from=1-2, to=1-3]
	\arrow["{\sigma_*}", from=1-3, to=1-4]
\end{tikzcd}
\end{center}
for a degenerable cube $A$, a morphism of cubes $\phi$ and a simplex morphism $\sigma$.
See also Definition~\ref{def:formal-pull-push} and the remarks thereafter.
It is clear that this functor $K$ defines a cosimplicial cubical set (by currying) which we also denote by $K$.
As a cubeset, $K_n$ is concrete with underlying set $K_n(0) = \su[n]$.

\begin{definition}\label{def:functor}
Define the functor $K^\ast\colon\cSet\to\sSet$ by
\[
    \cSet\ni X\mapsto K^\ast X\coloneqq \hom(K_{(-)}, X)\in\sSet.
\]
\end{definition}

By definition, the functor $K^\ast$ is right adjoint
and its left adjoint is given by the unique colimit preserving extension of $K$:
\begin{center}
\begin{tikzcd}
	\bbDelta & \cSet \\
	\sSet
	\arrow["K", from=1-1, to=1-2]
	\arrow[from=1-1, to=2-1]
	\arrow[""{name=0, anchor=center, inner sep=0}, "{K^\ast}", curve={height=-12pt}, from=1-2, to=2-1]
	\arrow[""{name=1, anchor=center, inner sep=0}, "{K_!}", from=2-1, to=1-2]
	\arrow["\dashv"{anchor=center, rotate=-46}, draw=none, from=1, to=0]
\end{tikzcd}
\end{center}

In the following, given a cubeset $X$ we denote by $X(m)$ its collection of $m$-cubes
and by $K^\ast_n X$ the collection of $n$-simplices of $K^\ast X$.
The next proposition makes the set of $n$-simplices of $K^\ast X$ explicit.

\begin{proposition}\label{prop:concrete-simplices}
Let $X$ be a concrete cubeset with the gluing property.
Then for a map $f\colon\su[n]\to X(0)$ the following assertions are equivalent.
\begin{enumerate}[(a)]
    \item For every $m\in\N$ and every $c\in K_n(m)$ the configuration
        \[
            f\circ c\colon\{0,1\}^m\to\su[n]\to X(0)
        \]
        is in $X(m)$.
    \item For every map $\sigma\colon[m]\to[n]$ in $\bbDelta$ the configuration
        \[
            f\circ\sigma_\ast\circ C_m\colon\{0,1\}^m\to\su[m]\to\su[n]\to X(0)
        \]
        is in $X(m)$.
    \item For every injection $i\colon[m]\to [n]$ in $\bbDelta$ the configurations
        \begin{align*}
            f\circ i_\ast\circ C_m&\colon\{0,1\}^m\to\su[m]\to\su[n]\to X(0) \\
            f\circ i_\ast\circ D_m&\colon\{0,1\}^m\to\su[m]\to\su[n]\to X(0)
        \end{align*}
        are in $X(m)$.
\end{enumerate}
In particular, $K^\ast_n X$ consists of all maps $f\colon\su[n]\to X(0)$
fulfilling one of the equivalent conditions above.
\end{proposition}

\begin{proof}
$(a)\implies (b)$:
The $m$-cube $C_m$ is degenerable, so $c = \sigma_\ast\circ C_m$ is in $K_n(m)$.

$(b)\implies (c)$:
It suffices to show that $f\circ i_\ast\circ D_m$ is in $X(m)$.
By Proposition~\ref{cor:pushforward-generates} we have that
$D_m = \sigma_{m,\ast}\circ C_{m+1}\circ\overline{\epsilon}_m$
where $\overline{\epsilon}_m$ is the inert inclusion corresponding to the face $[x_m=1]$.
Thus,
\[
      f\circ i_\ast\circ D_m
    = f\circ i_\ast\circ\sigma_{m,\ast}\circ C_{m+1}\circ\overline{\epsilon}_m
    = f\circ(i\circ\sigma_m)_\ast\circ C_{m+1}\circ\overline{\epsilon}_m
\]
is in $X(m)$ by assumption.

$(c)\implies (a)$: 
First consider the case where $c = A$ is a degenerable $n$-cube.
By Theorem~\ref{thm:C-D-glue-all-really}, $A$ can be obtained,
by gluing and taking cube morphisms, from $i_\ast C_m$ and $i_\ast D_m$ for injections $i\colon[m]\to[n]$.
Since $X$ has the gluing property, this means that $f\circ A$ is in $X(n)$.
The case where $c = \sigma_\ast\circ A$ for a morphism $\sigma\colon[m]\to[n]$
and $A$ a degenerable $m$-cube follows from Remark~\ref{rem:formal-gluing-push} and Proposition \ref{cor:pushforward-generates}.
The general case then follows from pullback along a cube morphism.

For the last assertion note that,
since $X$ and $K_n$ are concrete, a morphism of cubesets $f\in\hom(K_{n},X)$
is given by its underlying map $f_0\colon \su[n] = K_{n}(0)\to X(0)$ such that,
whenever $c\colon\{0,1\}^m\to\su[n]$ is in $K_n(m)$, the configuration $f_0\circ c$ is in $X(m)$,
see \cite[Lemma 2.3.4]{Bihlmaier2026}.
\end{proof}

\begin{remark}
In view of the above proposition, we could have also defined $K_n(m)$ to consist of all maps
$\{0,1\}^m\to\su[n]$ that admit a factorization
\begin{center}
\begin{tikzcd}[cramped]
	{\{0,1\}^m} & {\{0,1\}^k} & {\su[k]} & {\su[n]}
	\arrow["{\phi}", from=1-1, to=1-2]
	\arrow["C_k", from=1-2, to=1-3]
	\arrow["{\sigma_*}", from=1-3, to=1-4]
\end{tikzcd}
\end{center}
where $\phi$ is in $\bbox$ and $\sigma$ in $\bbDelta$.
Denote this cosimplicial cubeset by $K'$.

Recall that the inclusion of cubesets with the unique gluing property into all cubesets
admits a left adjoint $L$ by \cite[Lemma 3.1.25]{Bihlmaier2026}.
Moreover, the cubeset $LK_n$ consists of the gluing closure of $K_n$.
By Theorem~\ref{thm:C-D-glue-all-really} and Proposition~\ref{prop:degenerable-glue-C-D}
we have that $LK'_n = LK_n$, hence
\[
    \hom(K'_n,X) = \hom(LK'_n,X) = \hom(LK_n,X) = \hom(K_n,X)
\]
for every concrete cubeset $X$ with the gluing property.
\end{remark}

\begin{lemma}
We have that $X(0) = K^\ast_0X$.
In particular, the functor $K^\ast$ is faithful on concrete cubesets.
\end{lemma}

\begin{proof}
Since $\su[0]$ is a singleton, we have that $K_0$ is the terminal cubeset.
Thus, $K^\ast_0X = \hom(K_0,X) = \hom(\ast,X) = X(0)$.
Faithfulness of $K^\ast$ on concrete cubesets thus follows from \cite[Lemma 2.3.4]{Bihlmaier2026}.
\end{proof}

Now we are ready to prove our main theorem.

\begin{theorem}\label{thm:main}
Let $X$ be a concrete fibrant cubeset. Then the following assertions hold.
\begin{enumerate}[(i)]
	\item The simplicial set $K^\ast X$ is fibrant (a Kan complex).
	\item The set of ergodic components of $X$ is canonically isomorphic to the set of connected components of $K^\ast X$,
	and the isomorphism is natural in $X$.
    \item The $n$-th structure group of $X$ is canonically isomorphic to the $n$-th homotopy group of $K^\ast X$ for all $n\in\N$ and all vertices $x_0\in X(0)$,
    and the isomorphism is natural in $X$.
\end{enumerate}
\end{theorem}

For the definitions of fibrant cubesets, structure groups and ergodic components we refer to \cite{Bihlmaier2026}.
In the proof, we use the notation $\delta_i$, etc., for a subobject of $[n]$,
but note that they are in one-to-one correspondence with the non-empty subsets of $\{0,\ldots,n\}$,
so we also interchangeably use this notation (so that $[n]$ corresponds to $\id$, $\hat i$ to $\delta_i$, etc.).

\begin{proof}
ad (i).
The idea is to show:
\begin{enumerate}
    \item[(0)] For a horn $\lambda$ in $K^* X$ of the $n$-simplex with missing face $\delta_i$, we obtain a map $f\colon\su[n]\setminus\{\delta_i,\id\}\to X(0)$ such that most of the conditions of Proposition~\ref{prop:concrete-simplices} are fulfilled.
    \item[(1)] The formal cube $C_{n-1}^i$ corresponding to the missing face $\delta_i$ can be glued from the other formal cubes $C_{n-1}^j$ for $j\neq i$,
                corresponding to the other faces $\delta_j$ of the $n$-simplex.
	\item[(2a)] Using corner completion, we may also choose a point for $\hat i = \delta_i$ to construct a preliminary $D_{n-1}^i$, and thus all faces of $C_n$ exist. In particular, the ones without the label $\hat i$, which constitute a corner.
	        Using corner completion, construct $C_n$ and thus $f$ has a value on $\hat i$.
    \item[(2b)] Taking the corresponding face, construct (a now possibly different, but final) $D_{n-1}^i$.
    \item[(3)] All faces of $D_n$ that don't contain the full label $[n] = \id$, exist and thus we may complete to a $D_n$.
\end{enumerate}

Let $\lambda\colon\Lambda_n^i\to Y$ be an $n$-horn, $n\ge 2$, in $Y = K^\ast X$, $0\leq i\leq n$,
where $X$ is a fibrant concrete cubeset (and so in particular has the gluing property by~\cite[Proposition 3.1.23]{Bihlmaier2026}).
This means that we have a compatible family of ($n-1$)-simplices in $K^* X$, i.e., maps of cubesets $f_j\colon K_{n-1}\to X$, $j\neq i$, that agree on neighbouring ($n-2$)-faces.
We want to construct an $n$-simplex $f$ (a map $f\colon K_n\to X$ of cubesets) that has the $f_j$ as its $j$-th faces (i.e., $f\circ\delta_{j,\ast} = f_j$ on points),
so that it acts as a horn filling of $\lambda$.
So let us construct the map $f\colon\su[n]\to X(0)$ which, in order to induce a map of cubesets $f\colon K_n\to X$,
has to fulfill the condition from Proposition~\ref{prop:concrete-simplices}.

The compatibility with the $f_j$ already tells us the value of $f$ on every subobject of $[n]$, except for $\delta_i$ and $\id$,
and thus, given an injection $\iota\colon[k]\to[n]$ which is not the identity nor $\delta_i^{n-1}$, the maps
\begin{align*}
    f\circ \iota_\ast\circ C_k&\colon\{0,1\}^k\to\su[k]\to\su[n]\to X(0) \\
    f\circ \iota_\ast\circ D_k&\colon\{0,1\}^k\to\su[k]\to\su[n]\to X(0)
\end{align*}
are in $X(k)$ (here we can write $f$ for the map that doesn't yet exist, since $f$ is well-defined on the image of $\iota_\ast$).

The configuration $f\circ\delta_{i,\ast}\circ C_{n-1} = f\circ C_{n-1}^i$ is obtainable under gluing by the cubes $f\circ C_{n-1}^j$ for $j\neq i$,
see Remark~\ref{rem:glue-C-inverse}, and thus itself is an ($n-1$)-cube in $X$ by the gluing property.

Next we show that we can extend $f$ to $\su[n]\setminus\{\id\}$
by assigning a value on $\delta_i$ such that $f\circ C_n$ is an $n$-cube and $f\circ D_{n-1}^i$ an ($n-1$)-cube in $X$.
By Corollary~\ref{rem:D-faces}, every lower face $[x_j=0]$ of $D_{n-1}^i = \delta_{i,\ast} D_{n-1}$
can be glued from (the pushforward of) $D_{n-2}^\ell$, and so $f\circ D_{n-1}^i\rvert_{{\bbcor}^{n-1}}$ is an ($n-1$)-corner in $X$
(the expression with $f$ makes sense since the corner doesn't contain the full label $\delta_i$).
We can complete this to an ($n-1$)-cube and call this extended map $f'\colon\su[n]\setminus\{\id\}\to X(0)$,
as it will only be preliminary.
Now we can use the $f'\circ D_{n-1}^j$ and $f'\circ C_{n-1}^j$ to glue every face of $f\circ C_n$ by Theorem~\ref{thm:faces_of_Cn_generated}.
In particular, the $n$-corner of $C_n$ that does not hit the label $\delta_i$ exists in $X$
and can be completed to an $n$-cube with value $f\colon\su[n]\setminus\{\mathrm{id}\}\to X(0)$ such that $f\circ C_n$ is in $X(n)$.
Now we can go backwards and use the face $[x_n=1]$ or $[x_n=0]$ (depending on $i$) as well as backwards gluing with $f\circ D_{n-1}^j$, $j\neq i$, in the second case to obtain $f\circ D_{n-1}^i$ as an element of $X(n-1)$, see Remark~\ref{rem:C-faces-simple}.

Lastly, we again use Corollary~\ref{rem:D-faces} to see that every lower face of $D_n$ yields an ($n-1$)-cube under $f$ in $X$.
Thus we can complete the $n$-corner of $f\circ D_n$ to obtain a value of $f$ on the full label $[n]$.

Therefore we have now constructed a map $f\colon\su[n]\to X(0)$ such that
\begin{enumerate}
    \item[(1)] $f\circ\delta_{i,\ast}\circ C_{n-1} = f\circ C_{n-1}^i$ is in $X(n-1)$,
    \item[(2a)] $f\circ C_{n}$ is in $X(n)$,
    \item[(2b)] $f\circ\delta_{i,\ast}\circ D_{n-1} = f\circ D_{n-1}^i$ is in $X(n-1)$ and
    \item[(3)] $f\circ D_n$ is in $X(n)$.
\end{enumerate}
Together with the information we had earlier (0), we see that $f$ fulfills all properties needed to be an $n$-simplex (Proposition~\ref{prop:concrete-simplices}) and that it completes the $n$-horn.
Therefore, $K^\ast X$ is fibrant.

ad (ii).
The set $\pi_0(X)$ of ergodic components of $X$ consists of all points in $X(0)$ where one identifies two points $x$ and $y$ iff $(x,y)\in X(1)$ (which is an equivalence relation since $X$ is fibrant).
The set of connected components $\pi_0(K^\ast X)$ is all points in $K^\ast_0 X = X(0)$ with two points $x$ and $y$ identified iff there exists a path $p\in K^\ast_1 X$ connecting both (which also is an equivalence relation by fibrancy).
Since $K^\ast_1 X \cong X(1)\times_{X(0)}X(1)$, this means that $x\sim y$ iff there exists $z\in X(0)$ such that $(x,y),(z,y)\in X(1)$ which, by choosing $z=x$, is equivalent to $(x,y)\in X(1)$.
So the sets of connected components are isomorphic.
The situation is summarized by the following commutative diagram
\begin{center}
\begin{tikzcd}
	{X(1)} & {X(0)} & {\pi_0(X)} \\
	{X(1)\times_{X(0)}X(1)} & {X(0)} & {\pi_0(K^\ast X)}
	\arrow[shift left, from=1-1, to=1-2]
	\arrow[shift right, from=1-1, to=1-2]
	\arrow["\Delta"', from=1-1, to=2-1]
	\arrow[from=1-2, to=1-3]
	\arrow[from=1-2, to=2-2]
	\arrow[dashed, from=1-3, to=2-3]
	\arrow[shift left, from=2-1, to=2-2]
	\arrow[shift right, from=2-1, to=2-2]
	\arrow[from=2-2, to=2-3]
\end{tikzcd}
\end{center}
where both rows are coequalizer diagrams and the right vertical arrow,
induced by the universal property of the coequalizer,
is exactly the isomorphism $\pi_0(X)\simeq\pi_0(K^\ast X)$.
This also shows naturality in $X$.

ad (iii).
Let $n\in\N$, $x_0\in X(0)$ and $Y = K^\ast X$. Recall the definition of the structure group $A_n(X,x_0)$, see \cite[Construction 3.3.8]{Bihlmaier2026}.
Let $\bbcor_n(X,x_0)$ be the set of all based $n$-cubes
\begin{center}
\begin{tikzcd}
	{\bbcor_n(X,x_0)} & \ast \\
	{\hom({\bbox}^n,X)} & {\hom({\bbcor}^n,X)}
	\arrow[from=1-1, to=1-2]
	\arrow[from=1-1, to=2-1]
	\arrow["{x_0}", from=1-2, to=2-2]
	\arrow[from=2-1, to=2-2]
\end{tikzcd}.
\end{center}
On the other hand, let $\tilde{\pi}_n(K^\ast X,x_0)$ be the set of all $n$-simplices with constant boundary $x_0$, i.e., the pullback
\begin{center}
\begin{tikzcd}[cramped]
	{\tilde{\pi}_n(Y,x_0)} & \ast \\
	{\hom({\bbDelta}^n,Y)} & {\hom(\partial{\bbDelta}^n,Y)} \\
	{\hom(K_n,X)} & {\hom(\varinjlim\limits_{i<n}K_i,X)}
	\arrow[from=1-1, to=1-2]
	\arrow[from=1-1, to=2-1]
	\arrow["{{x_0}}", from=1-2, to=2-2]
	\arrow[from=2-1, to=2-2]
	\arrow["\simeq"', from=2-1, to=3-1]
	\arrow["\simeq", from=2-2, to=3-2]
	\arrow[from=3-1, to=3-2]
\end{tikzcd}.
\end{center}
Now, given such an $n$-sphere $s\colon\su[n]\to X(0)$, one sees that being constant on the boundary
implies that $s$ has value $x_0$ on every label that is not the full label $[n]$.
The condition from Proposition~\ref{prop:concrete-simplices} on the $C_m$, $m\leq n$, and on $D_k$, $k<n$, becomes trivial.
The only non-trivial condition is the one for $D_n$ which has the full label $\id = [n]$ in exactly one vertex.
This means that $\tilde{\pi}_n(Y,x_0)$ consists of all $n$-cubes whose $n$-corner is constant $x_0$.
Therefore, $\bbcor_n(X,x_0)\simeq\tilde{\pi}_n(Y,x_0)$.

The $n$-structure group $A_n(X,x_0)$ of $X$ is the quotient of $\bbcor_n(X,x_0)$ that identifies two $n$-cubes
$c$ and $d$ iff the cube that has $c$ on the face $[x_0=0]$ and $d$ on the face $[x_0=1]$ is an ($n+1$)-cube in $X$,
see \cite[Lemma 3.3.9]{Bihlmaier2026}.%
\footnote{We denote the structure group by the letter $A$ here, and not $\pi$ as in \cite{Bihlmaier2026},
to distinguish it visually from the homotopy group of the simplicial set,
which also is denoted $\pi$.}
On the other hand, the $n$-th homotopy group $\pi_n(Y,x_0)$ is the set $\tilde{\pi}_n(Y,x_0)$ modulo the equivalence relation
where one identifies two $n$-spheres $s$ and $t$ if there exists an ($n+1$)-simplex $f$ in $Y$ with boundary $(\id,\ldots,\id,s,t)$
(meaning that $f\circ\delta_i = x_0$ for all $i\geq 2$, $f\circ\delta_n = s$ and $f\circ\delta_1 = s$),
see \cite[Lemma 7.4]{Goerss1999}.

Again using the description of $Y$, $f$ corresponds uniquely to a map $f\colon\su[n+1]\to X(0)$ of cubesets whose labels of length $\leq n$ are $x_0$
and likewise $f(\delta_i) = x_0$ for $i<n$ (and $f(\delta_n) = s(\id)$, $f(\delta_{n+1}) = t(\id)$).
Thus, we can reduce the existence of such an $f$ to the question whether $f\circ D_n^n$, $f\circ D_n^{n+1}$, $f\circ C_{n+1}$ and $f\circ D_{n+1}$ are in $X(n)$, resp. $X(n+1)$.
But $c=f\circ D_n^n$ and $d=f\circ D_n^{n+1}$ are $s$ and $t$ under the isomorphism $\bbcor_n(X,x_0)\simeq\tilde{\pi}_n(Y,x_0)$ and thus exist.
Since $X$ is fibrant and we have enough faces of $D_{n+1}$ to use corner completion by Corollary~\ref{rem:D-faces}, the existence of $f$ only depends on $C_{n+1}$.
But note that the labels $\delta_n$ and $\delta_{n+1}$ appear precisely once in $C_{n+1}$ next to each other, and every other label has length at most $n$.
So $f\circ C_{n+1}$ is (up to cube automorphisms) precisely the configuration obtained from $c$ and $d$ put next to each other.
Thus, $s\sim t$ iff $c\sim d$.
This induces a (natural) bijection $A_n(X,x_0)\simeq\pi_n(Y,x_0)$.

Lastly, we show that this bijection also is a group homomorphism.
Let $c,d\in\bbcor_n(X,x_0)$ with corresponding $n$-spheres $s$ and $t$.
Recall that their product in $\pi_n(Y,x_0)$ is given by the
(equivalence class of the) $n$-th face $\delta_n$
of a horn filling of the horn $\lambda\colon\Lambda^{n+1}_n\to Y$ with boundary
$\lambda(\delta_i) = x_0$ for $i\leq n-2$, $\lambda(\delta_{n-1}) = s$ and $\lambda(\delta_{n+1}) = t$,
see \cite[Lemma 7.1]{Goerss1999}.
This horn filling amounts to the datum of a map $f\colon\su[n+1]\to X(0)$ with
$f(s) = x_0$ for each subobject $s$ of size $\leq n$,
$f(\delta_i) = x_0$ for all $i\leq n-2$, $f(\delta_{n-1}) = s$ and $f(\delta_{n+1}) = t$
such that $f$ fulfills the usual compatibilities on the $C_k$- and $D_k$-cubes.
The resulting product of $s$ and $t$ is thus (the equivalence class of) $f(\delta_n)$.
Now the ($n+1$)-cube $f\circ C_{n+1}$ has the value $c(1)$ at $(1,\ldots,1,0,0)$,
the value $d(1)$ at $(1,\ldots,1)$, the value $f(\delta_n)$ at $(1,\ldots,1,0)$
and $x_0$ everywhere else.
After appropriately flipping and permuting coordinates,
the cube $f\circ C_{n+1}$ is the datum of the multiplication $\overline{c*d}$,
see \cite[Construction 3.3.11]{Bihlmaier2026}.
Thus, the bijection $A_n(X,x_0)\simeq\pi_n(Y,x_0)$ is a group homomorphism.
\end{proof}

\begin{remark}
The proof can be strengthened to obtain the following:
If a map $f\colon X\to Y$ between concrete cubesets with gluing property
is a fibration, then the induced map $K^\ast f\colon K^\ast X\to K^\ast Y$ of simplicial sets
is a Kan fibration.
\end{remark}

Moreover, the functor $K^\ast$ maps $n$-step cubesets to $n$-truncated simplicial sets.

\begin{definition}
A Kan complex $X$ is \emph{$n$-truncated} if the map
\[
    \hom(\bbDelta^k,X)\to\hom(\partial\bbDelta^k,X)
\]
is surjective for all $k\geq n+2$.
\end{definition}

\begin{corollary}
Let $X$ be an $n$-step fibrant cubeset. Then $K^\ast X$ is $n$-truncated.
\end{corollary}

\begin{proof}
This follows from \cite[Proposition 3.3.13]{Bihlmaier2026},
the theorem and \cite[Proposition 054V]{Kerodon2026}.
\end{proof}

\begin{remark}
The proof of the theorem also shows that if $X$ is $n$-step,
then the lift of the horn $\lambda$ in $K^* X$ to a $k$-simplex is unique for $k\geq n+1$.
In particular, $K^\ast X$ is an $n$-groupoid, see \cite[Definition 053M]{Kerodon2026}.
This yields another proof that $K^\ast X$ is $n$-truncated:
every $n$-groupoid is already $n$-truncated.
\end{remark}

Also, ergodicity is translated to connectivity.

\begin{definition}
A Kan complex $X$ is \emph{$n$-connective} (or \emph{$(n-1)$-connected})
if for every basepoint $x_0\in X$ and every $0\leq k<n$
the homotopy group (set of connected components for $k=0$) $\pi_k(X,x_0)$ is a singleton.
\end{definition}

\begin{corollary}
If $X$ is an $n$-ergodic fibrant concrete cubeset, then $K^\ast X$ is $n$-connective.
\end{corollary}

\section{Condensed Enrichment}\label{sec:enrich}

We refer to \cite[Section 4]{Bihlmaier2026} for terminology and definitions in this section, and to \cite{Scholze2026, Scholze2026a, Clausen2026, Bihlmaier2025} for the basics on condensed mathematics.
In particular, we implicitly fix a strong limit cardinal $\kappa$.

Much of the structure theory of nilspaces is developed not purely in cube\emph{sets}
but rather with an additional topological structure,
mostly compact metrizable Hausdorff spaces.
In \cite[Section 4]{Bihlmaier2026} we obtained this setting by tensoring the category of cubesets
with the category of ($\kappa$-)condensed sets in $\PrL$,
which ultimately led to the classical weak structure theorem.
This also allows us to formulate a topologically enriched version of the transition functor.

\begin{proposition}\label{prop:computation-K-pointwise}
The adjunction $K_!\dashv K^\ast$ induces an adjunction
\[K_!\otimes\id\colon\sSet\otimes\CondSet\rightleftarrows\cSet\otimes\CondSet\lon K^\ast\otimes\id.\]
The right adjoint $K^\ast\otimes\id$ of this adjunction makes the diagrams
\begin{center}
\begin{tikzcd}
	{\cSet\otimes\CondSet} & {\sSet\otimes\CondSet} & {\cSet\otimes\CondSet} & {\sSet\otimes\CondSet} \\
	{\Fun^\times(\extr_\kappa^\op,\cSet)} & {\Fun^\times(\extr_\kappa^\op,\sSet)} & {\Fun({\bbox}^\op,\CondSet)} & {\Fun(\bbDelta^\op,\CondSet)}
	\arrow["{K^\ast\otimes\id}", from=1-1, to=1-2]
	\arrow["\simeq"', from=1-1, to=2-1]
	\arrow["\simeq", from=1-2, to=2-2]
	\arrow["{K^\ast\circ-}"', from=2-1, to=2-2]
    \arrow["{K^\ast\otimes\id}", from=1-3, to=1-4]
	\arrow["\simeq"', from=1-3, to=2-3]
	\arrow["\simeq", from=1-4, to=2-4]
	\arrow["{X\mapsto\hom((K_{(-)})_\ud,X)}"'{yshift=-4pt}, from=2-3, to=2-4]
\end{tikzcd}
\end{center}
commutative.
Explicitly, this means for a condensed cubeset $X$ we have that
\[
    (K^\ast X)(E) = K^\ast(X(E)) = \hom_{\cSet}(K_{(-)},X(E))
    \quad\text{and}\quad
    K^\ast_n X = \hom_{\cCond}((K_n)_\ud,X)
\]
for $E\in\extr_\kappa$ and $[n]\in\bbDelta$.
\end{proposition}

Here, given a cubeset $X$, $X_\ud$ denotes the corresponding discrete condensed cubeset,
see~\cite[Definition 4.2.11]{Bihlmaier2026}.

\begin{proof}
The induced adjunction is a direct consequence of functoriality of the tensor product,
see \cite[Section 4.8.1]{Lurie2017}.
The diagram on the left commutes by \cite[Remark 071F]{Kerodon2026}.
This directly implies the first equality $(K^\ast X)(E) = K^\ast(X(E))$
under the equivalence $\CondSet\otimes \sSet\simeq\Fun^\times(\extr_\kappa^\op,\sSet)$
since $K^\ast$ is computed pointwise.

Under the equivalence $\cSet\otimes\CondSet\simeq\Fun^\uR(\cSet^\op,\CondSet)$,
the right adjoint $K^\ast\otimes\id$ is given by precomposition with
$K_!^\op\colon\sSet^\op\to\cSet^\op$, i.e. $X\mapsto X\circ K_!^\op$,
see \cite[Remark 071F]{Kerodon2026}.
Furthermore, the equivalence $\Fun^\uR(\cSet^\op,\CondSet)\simeq\Fun({\bbox}^\op,\CondSet)$
is given by restriction and right Kan extension, respectively,
and thus, under this equivalence, the map
$\Fun({\bbox}^\op,\CondSet)\to\Fun(\bbDelta^\op,\CondSet)$
is given by right Kan extending $X$ to a limit preserving functor $\cSet^\op\to\CondSet$,
then precomposing with $K_!^\op$ and lastly restricting to $\bbDelta^\op$.
Given $[n]\in\bbDelta$, we write $K_n = \varinjlim_{i\in\cI}{\bbox}^i$ as the colimit of representables.
Then we have for a functor $X\colon{\bbox}^\op\to\CondSet$ that
\[
    (\Ran X\circ K_!^\op)([n]) = (\Ran X)(K_!^\op([n])) = (\Ran X)(K_n) = \varprojlim_{i\in\cI}X\left({\bbox}^i\right).
\]
Now by the enriched Yoneda lemma for condensed cubesets~\cite[Lemma 4.2.12]{Bihlmaier2026}
we have that $X({\bbox}^i) = \hom({\bbox}^i_\ud,X)$ where $(-)_\ud$
denotes the (topologically) discrete condensed cubeset, see \cite[Lemma 4.2.10]{Bihlmaier2026}.
Since both the condensed enriched $\hom$-functor and $(-)_\ud$ preserve colimits in the first variable,
we obtain that
\[
    \varprojlim_{i\in\cI}X\left({\bbox}^i\right) = \varprojlim_{i\in\cI}\hom\left({\bbox}^i_\ud,X\right)
    = \hom\left(\varinjlim_{i\in\cI}\left({\bbox}^i_\ud\right),X\right)
    = \hom\left(\left(\varinjlim_{i\in\cI}{\bbox}^i\right)_\ud,X\right)
    = \hom\left((K_n)_\ud,X\right).
\]
\end{proof}

\begin{corollary}\label{thm:main_enriched}
Let $X$ be a concrete fibrant condensed cubeset. Then the following assertions hold.
\begin{enumerate}[(i)]
    \item The simplicial condensed set $K^\ast X$ is a Kan complex (i.e., every simplicial set $K^*X(E)$ for $E\in\extr_\kappa$ is a Kan complex).
    \item The condensed set of ergodic components of $X$
            agrees with the condensed set of connected components of $K^\ast X$.
    \item For every basepoint $x_0\in X$, the condensed structure groups of $X$
            are isomorphic to the condensed homotopy groups of $K^\ast X$.
\end{enumerate}
Moreover, if $X$ is qcqs, then so is $K^\ast X$.
\end{corollary}

\begin{proof}
The claims (i)-(iii) directly follow from combining Theorem~\ref{thm:main}
with Proposition~\ref{prop:computation-K-pointwise}.

The last claim follows from the fact that qcqs condensed sets are stable under limits
\cite[Proposition 2.5.49]{Bihlmaier2025}:
we have that $K^\ast_n X = \varprojlim_{i\in\cI} X({\bbox}^i)$
by the proof of Proposition~\ref{prop:computation-K-pointwise}
and each $X({\bbox}^i)$ is qcqs.
\end{proof}

Note that in (iii), by the condensed homotopy groups we mean the condensed groups obtained by pointwise taking homotopy groups.
A priori, this might not be the same as the internal homotopy groups of the corresponding condensed anima.
We close this section by showing that these constructions agree, leading to the homotopy coherent variant of the previous results.
We are convinced that this is well-known to experts, but we could not find an explicit reference in the literature.
First, let us recall the necessary constructions and definitions.

In \cite[Section 4.1]{Bihlmaier2026} we have referred to the $\infty$-category $\Ani$ of anima
just as the tensor unit of $\PrL$,
defined via its universal property that $\Fun^\uR(\Ani^\op,\cC) \simeq \cC$ for every presentable $\infty$-category $\cC$.
Let us recall a model-dependent description of $\Ani$ that is based on Kan complexes.

\begin{construction}
Let $\mathrm{Kan}$ be the full subcategory of $\sSet$ consisting of all Kan complexes.
Since $\sSet$ is enriched over itself
(its internal Hom has $\hom(X\times\bbDelta^n,Y)$ as its set of $n$-simplices)
and this internal Hom is a Kan complex whenever the codomain is,
the category $\mathrm{Kan}$ is enriched over itself.
Thus, taking the homotopy coherent nerve $\uN^\mathrm{hc}(\mathrm{Kan})$ (\cite[Definition 1.1.5.5]{Lurie2009}, called simplicial nerve there)
of the category of Kan complexes gives an $\infty$-category by \cite[Proposition 1.1.5.10]{Lurie2009},
which turns out to be the $\infty$-category $\Ani$, see \cite[Corollary 5.1.5.8]{Lurie2009}.
Note that there is a functor $\mathrm{Kan}\to\Ani$ of $\infty$-categories
via the usual embedding of 1-categories into $\infty$-categories,
see \cite[Remark 1.1.5.8]{Lurie2009} and \cite[Remark 00KV]{Kerodon2026}.
This functor preserves arbitrary products by \cite[Example 7.6.1.19]{Kerodon2026}.
\end{construction}

Going back to the cubical setting, the functor $K^\ast\colon\cSet\to\sSet$
associates to each concrete fibrant cubeset $X$ a Kan complex $K^\ast X$
and thus restricts to a functor $K^\ast\colon\ccSetnil\to\mathrm{Kan}$
preserving finite products (since both $\ccSetnil$ and $\mathrm{Kan}$ are closed under finite products
in $\cSet$ and $\sSet$, respectively).
By composition, we obtain a finite product preserving functor $K^\ast\colon\ccSetnil\to\Ani$,
giving each concrete fibrant cubeset its corresponding homotopy type.
Pushforward of this functor thus yields a functor
\[
    \Fun^\times(\extr_\kappa^\op,\ccSetnil)\to\Fun^\times(\extr_\kappa^\op,\Ani)
\]
where the left hand side can be identified with the full subcategory of condensed cubesets
which are concrete and fibrant (see \cite[Sections 4.3 and 4.4]{Bihlmaier2026})
and the right hand side can be identified with the category of condensed anima
by \cite[Proposition 4.2.2]{Bihlmaier2026}.

We now want to compare the (condensed) structure groups of a condensed cubeset
with the (condensed) homotopy groups of the corresponding condensed anima.

\begin{construction}
Given an $\infty$-topos $\cX$, there is an internal notion of homotopy groups in $\cX$,
see \cite[Section 6.5.1]{Lurie2009}.
Note that $\cX$ has all limits, and thus is cotensored over $\Ani$ by \cite[Remark 5.5.2.6]{Lurie2009}.
This means that for every simplicial set $K$ and $X\in\cX$ there is an object $X^K$
that corepresents the functor
\[
    Y\mapsto \hom_{\mathrm{ho}\Ani}(K,\hom(Y,X)).
\]
Here, $\mathrm{ho}\Ani$ is the homotopy category of anima,
which is equivalent to the (classical) homotopy category of CW-complexes.
This means that there is a natural isomorphism in $\mathrm{ho}\Ani$
\[
    \hom(Y,X^K)\simeq\hom_{\mathrm{ho}\Ani}(K,\hom(Y,X)).
\]
Essentially by definition, the object $X^K$ is obtained by taking the limit of the constant diagram
$K\to\cX$ with value $X$, see \cite[Corollary 4.4.4.9]{Lurie2009}.
Given a point of the $n$-sphere $\uS^n = \partial\bbDelta^{n+1}$, the map $\ast\to\uS^n$ of simplicial sets
induces by pullback a map $s\colon X^{\uS^n}\to X$.
Then the $n$-th homotopy group of $X$ is given by $\pi_n(X) = \tau_0(s)\in\mathrm{Disc}(\cX_{/X})$.
Here, $\tau_0\colon\cX_{/X}\to\mathrm{Disc}(\cX_{/X})$ denotes the $0$-truncation.

The projection $\uS^n\to\ast$ gives a basepoint of $\pi_n(X)$
and the comultiplication $\uS^n\to\uS^n\,\vee\,\uS^n$ induces a multiplication on $\pi_n(X)$ for $n\ge 1$.
It then can be shown that $\pi_n(X)$ is a group object in $\mathrm{Disc}(\cX_{/X})$
which is abelian for $n\geq 2$, see \cite[Section 6.5.1]{Lurie2009}.
\end{construction}

This construction seems very abstract at first,
but in many cases, and especially those relevant for us,
we can compute the homotopy groups of $\cX$.
The following example and lemmas explain how.

\begin{example}
This notion of homotopy group recovers the classical definition of homotopy group of a Kan complex.
Explicitly, if $K$ is a Kan complex with point $x_0\colon\ast\to K$,
interpreted as an anima,
then $x_0^\ast\pi_n(K)$ is a group object in $\set$ (so just a classical group)
that is given by the classical (simplicial) homotopy group $\pi_n(K,x_0)$, see \cite[Remark 6.5.1.5]{Lurie2009}.
\end{example}

\begin{lemma}
If $\cX = \PSh(\cC,\Ani) = \PSh(\cC)$ is the category of presheaves on a small category $\cC$,
then the homotopy groups $\pi_n(X)\in\mathrm{Disc}(\PSh(\cC)_{/X})$ of a presheaf $X$ are given
by the pointwise homotopy groups $\pi_n(X(c))\in\mathrm{Disc}(\Ani_{/X(c)})$.
\end{lemma}

\begin{proof}
Note that evaluation $\PSh(\cC)\to\Ani$ at an object $c\in\cC$ commutes with limits and colimits
since both are computed pointwise by \cite[Corollary 5.1.2.3]{Lurie2009}.
Thus, given a presheaf $X\in\PSh(\cC)$, the diagram
\begin{center}
\begin{tikzcd}
	{\PSh(\cC)_{/X}} & {\Ani_{/X(c)}} \\
	{\mathrm{Disc}(\PSh(\cC)_{/X})} & {\mathrm{Disc}(\Ani_{/X(c)})}
	\arrow["{\ev_c}", from=1-1, to=1-2]
	\arrow["{\tau_0}"', from=1-1, to=2-1]
	\arrow["{\tau_0}", from=1-2, to=2-2]
	\arrow["{\ev_c}"', from=2-1, to=2-2]
\end{tikzcd}
\end{center}
commutes (up to equivalence) by \cite[Proposition 5.5.6.28]{Lurie2009}.
Since limits (and thus, cotensors) of presheaves are computed pointwise,
the map $s\colon X^{\uS^n}\to X$ is given pointwise by $s_c\colon X(c)^{\uS^n}\to X(c)$
(and we even have that $X(c)^{\uS^n} = \hom(\uS^n,X(c))$).
Thus, we conclude
\[
    \pi_n(X)(c) = \tau_0(s)(c) = \tau_0(s_c) = \pi_n(X(c)).\qedhere
\]
\end{proof}

\begin{lemma}
Let $X$ be a condensed anima.
Under the equivalence $\CondAni\simeq\Fun^\times(\extr_\kappa^\op,\Ani)$,
the (internal) homotopy groups $\pi_n(X)$ are given by the pointwise homotopy groups
$\pi_n(X(E))\in\mathrm{Disc}(\Ani_{/X(E)})$.
\end{lemma}

\begin{proof}
This follows from the previous lemma
noting that truncation commutes with finite products by \cite[Lemma 6.5.1.2]{Lurie2009} and cotensoring as well, being a limit.
Thus, the (pointwise) homotopy group preserves finite products and restricts to $\Fun^\times(\extr_\kappa^\op,\Ani)$.
\end{proof}

This concludes our comparison of the structure group with the homotopy group.

\begin{corollary}
Let $X$ be a concrete fibrant condensed cubeset with point $x_0\colon\ast\to X$
and $Y = K^\ast X$ the corresponding condensed anima.
Then the $n$-th condensed structure group $\pi_n(X,x_0)$ of $X$
is naturally isomorphic to the condensed homotopy group $x_0^\ast\pi_n(Y)$.
\end{corollary}

\printbibliography[heading=bibintoc]
\end{document}